\documentclass[reqno]{amsart}
\usepackage{amsmath}
\usepackage{amssymb}
\usepackage{titlesec} 
\usepackage{graphics}
\usepackage{graphicx}
\usepackage{tikz}
\usepackage{hyperref}
\usepackage{hyperref}

\titleformat{\section}[hang]{ \bf}
  {\thesection}{0.5em}{\large}

\titleformat{\subsection}[runin]{\bf}
{\thesubsection}{0.5em}{}

\newcommand{\Real}{\mathrm{Re}\,}

\newcommand{\aff}{\mathrm{aff}\,}
\def\argmin{ \mathop{{\rm argmin}}}
\def\argmax{ \mathop{{\rm argmax}}}

\newcommand{\co}{\mathrm{conv}\,}
\newcommand{\cl}{\mathrm{cl}\,}
\newcommand{\clco}{\mathrm{\overline{conv}}\,}
\newcommand{\diag}{\mathrm{diag}\,}
\newcommand{\dom}{\mathrm{dom}\,}
\newcommand{\epi}{\mathrm{epi}\,}

\newcommand{\gph}{\mathrm{gph}\,}
\newcommand{\inter}{\mathrm{int}\,}

\newcommand{\ri}{\mathrm{ri}\,}

\newcommand{\rge}{\mathrm{rge}\,}

\newcommand{\hzn}{\mathrm{hzn}\,}
\newcommand{\lin}{\mathrm{span}\,}

\newcommand{\p}{\partial}

\newcommand{\R}{\mathbb{R}}

\newcommand{\bS}{\mathbb{S}}

\newcommand{\lev}[2]{\mathrm{lev}_{#1}(#2)}

\newcommand{\rank}{\mathrm{rank}\,}
\newcommand{\rbar}{\overline{\mathbb R}}

\newcommand{\rp}{\mathbb R\cup\{+\infty\}}
\newcommand{\bR}{\mathbb{R}}

\def\bS{\mathbb{S}}

\def\bH{\mathbb{H}}

\newcommand{\tr}{\mathrm{tr}\,}

\newcommand{\cone}{\mathrm{cone}}

\newcommand{\bN}{\mathbb{N}}
\newcommand{\bE}{\mathbb{E}}
\newcommand{\bG}{\mathbb{G}}

\newcommand{\bC}{\mathbb{C}}
\newcommand{\map}[3]{#1 :#2\rightarrow #3}
\newcommand{\ip}[2]{\left\langle #1,\, #2\right\rangle}
\newcommand{\half}{\frac{1}{2}}

\newcommand{\sqnorm}[1]{\left\Vert #1\right\Vert_2^2}

\newcommand{\set}[2]{\left\{#1\,\left\vert\; #2\right.\right\}}

\newcommand{\clcone}[1]{\overline{\mathrm{cone}}\; #1}

\newcommand{\tX}{{\widetilde X}}

\newcommand{\tU}{{\widetilde U}}
\newcommand{\tV}{{\widetilde V}}

\newcommand{\tC}{{\widetilde C}}
\newcommand{\tD}{{\widetilde D}}

\def\tmu{{\tilde\mu}}

\newcommand{\alf}{\alpha}
\newcommand{\gam}{\gamma}

\newcommand{\sig}{\sigma}

\newcommand{\lam}{\lambda}

\newcommand{\tgam}{{\tilde{\gam}}}

\newcommand{\cD}{\mathcal{Q}}

\newcommand{\cF}{\mathcal{F}}

\newcommand{\st}{\ \mbox{ such that }\ }

\newcommand{\Epi}[2]{{#1}\mbox{-}\mathrm{epi}{\,#2}}

\definecolor{armygreen}{rgb}{0.29, 0.33, 0.13}
\definecolor{forest}{rgb}{0.0, 0.5, 0.0}

\newtheorem{proposition}{Proposition}
\newtheorem{theorem}[proposition]{Theorem}
\newtheorem{corollary}[proposition]{Corollary}
\newtheorem{lemma}[proposition]{Lemma}
\newtheorem{definition}[proposition]{Definition}

\newtheorem{example}[proposition]{Example}
\newtheorem{remark}[proposition]{Remark}

\begin{document}

\title[A study of convex convex-composite functions]{A study of convex convex-composite functions via infimal convolution with applications}

\author[J.V. Burke]{James V. Burke}
\address{Department of Mathematics, University of Washington, 7521 - 30th Ave. N.E., Seattle, WA 98115, USA }
\email{jvburke01@gmail.com}

\author[T. Hoheisel]{Tim Hoheisel}
\address{Department of Mathematics and Statistics, McGill University, 805 Sherbrooke St West, Room 1114, Montr\'eal, Qu\'ebec, Canada H3A 0B9}
\email{tim.hoheisel@mcgill.ca}
 
\author[Q.V. Nguyen]{Quang V. Nguyen}
\address{Department of Mathematics and Statistics, McGill University, 805 Sherbrooke St West, Montr\'eal, Qu\'ebec, Canada H3A 0B9}
\email{van.q.nguyen@mcgill.ca}

\dedicatory{In memory of Jonathan M. Borwein}

\begin{abstract}
In this note we provide a full conjugacy and subdifferential calculus  
for convex convex-composite functions  in finite-dimensional space. Our  approach,   based on infimal convolution and cone-convexity, is straightforward and yields the desired results under a    verifiable Slater-type condition, with relaxed monotonicity   and without lower semicontinuity   assumptions on the functions in play. The versatility of our findings is illustrated by a series of applications  in  optimization and   matrix analysis, including conic programming, matrix-fractional,  variational Gram,    and spectral   functions.
\end{abstract}



\maketitle

%


%
%
%

\section{Introduction} 

\noindent
Convex-composite optimization is a class of nonsmooth,  nonconvex optimization problems
that captures a wide variety of optimization models studied in modern optimization
practice and theory including
nonlinear programming (NLP), nonconvex minimax problems, nonconvex system
identification, inverse problems and nonlinear filtering
as well as most nonconvex problems in large-scale data analysis  and machine
learning, e.g. \cite{drusvyatskiyefficiency, DuR2017arXiv170502356D,duchi2017stochastic}. Given  two Euclidean spaces $\bE_1$ and $\bE_2$, convex-composite problems take the form
\begin{equation}\label{eq:cvx-comp}
\min_{x\in\bE_1} \Phi(x):=f(x)+g(F(x)),
\end{equation}
where $\map{F}{\bE_1}{\bE_2}$ is continuous, usually smooth, while
$\map{f}{\bE_1}{\R\cup\{+\infty\}}$ and $\map{g}{\bE_2}{\R\cup\{+\infty\}}$
are closed, proper, convex functions. 
The function $g$ encodes the modelling framework
such as an NLP or an inverse problem and the function $f$ is a regularizer used to induce
further properties on the solution such as sparsity, smoothness, or a domain restriction. The function $F$ is the 
functional data associated with a specific instance of the problem.
In many treatments the function $f$ is often subsumed into the function $g$. This adds simplicity and is useful 
when establishing a number of theoretical results, a fact that we  also exploit in  our study.

The study of the  convex-composite class began in the 
1970's \cite{powell1978algorithms, powell1978fast}
with major contributions occurring in the
1980's  \cite{burke1985descent,burke1987second,kawasaki1988second,rockafellar1985extensions,rockafellar1989second, wright1987local, yuan1985superlinear} 
and 
1990's \cite{BF 95, burke1992optimality, deng1996uniqueness,RoW 98}. 
Recently there has been a resurgence of interest in the foundations of
this problem class due to its importance for many problems in modern optimization,
machine learning, and large scale data analysis 
\cite{BuE19, BuH 13, cibulka2016strong,cui2018composite, davis2018stochastic, drusvyatskiy2018error, drusvyatskiyefficiency, DuR2017arXiv170502356D,duchi2017stochastic, lewis2016proximal}.

In this note we study the case where  $\Phi$ is convex.   Although the convex-composite setting was proposed as a  structure approach to nonconvex problems, understanding the convex case is   important. This is   illustrated in this paper by  the applications in different areas of convex analysis and optimization which underline the relevance of our study.

So, when is $\Phi$ convex?
We answer this with a simple result based on \cite[Lemma 5.1]{Bur91}. 
The statement
and proof are elementary, and so we
give this result now since it motivates much of our subsequent investigation.

\begin{theorem}\label{thm:cvx cvxcomp}
Let the assumptions in \eqref{eq:cvx-comp} hold.
Let $\alf\in \R$ be such that 
$\lev{g}{\alf}:=\set{y\in\bE_2}{g(y)\le\alf}\ne\emptyset$ and define the horizon cone of $g$
by $$\hzn{g}:=\set{w\in \bE_2}{\forall t\geq 0:\;y+tw\in\lev{g}{\alf}\;\; (y\in\lev{g}{\alf})}.$$
If $F$ is convex with respect to 
$-\hzn{g}$, i.e. 
\[
F((1-\lam)x_1+\lam x_2)-[  (1-\lam)F(x_1)+\lam F(x_2)]\in \hzn{g}
\quad(x_1,x_2\in\bE_1,\ \lam\in[0,1] ),
\]
then $\Phi$ is a convex function.
\end{theorem}

\begin{proof}
First recall that, by \cite[Theorem 14.2]{Roc 70},  $\hzn{g}$ is the polar cone of $\overline{\cone}(\dom g^*)$, 
where $g^*$ is the convex conjugate of $g$, cf. \eqref{eq:Conjugate} below.
Hence, $F$ is convex with respect to $-\hzn{g}$ if and only if
\[
\ip{z}{F((1-\lam)x_1+\lam x_2)}\le \ip{z}{(1-\lam)F(x_1)+\lam F(x_2)}\quad (z\in \dom g^*),
\] 
or equivalently,
\[
\ip{z}{F((1-\lam)x_1+\lam x_2)}-g^*(z)\le 
\ip{z}{(1-\lam)F(x_1)+\lam F(x_2)}-g^*(z)\;(z\in \dom g^*).
\] 
Since   $g=g^{**}$ (as $g$ is closed, proper, convex) this proves the result. 
\end{proof}

\noindent
This result illustrates the key features of our study:
(i)
the convexity of $F$ with respect to a cone $K$ and 
(ii) 
the  convexity of  $g$ and its monotonicity  with respect to $K$.
This structure is present in the most familiar result concerning the convexity
of compositions which states that  if $\bE_2=\R$ and $g$ is nondecreasing and  convex on 
its domain and $F$ is convex, then $\Phi$ is convex. In this case, clearly, $K=\R_+$.
 In particular, since $g$ is nondecreasing,
$\hzn{g}\subset \R_-$.

These two ideas, cone convexity and monotonicity, are 
certainly not new and are fundamental in convex analysis. Hence 
we now describe some of the most pertinent references for this study:
Borwein \cite{Bor 74} 
pursued an ambitious program of extending
most of convex analysis to cone convex functions including conjugacy, subdifferential analysis,
and duality, laying out much of the groundwork.  Kusraev and Kutateladze  \cite{KK95} take this idea to an even more general setting by considering {\em convex operators} with values in arbitrary ordered vectors spaces. Combari, Laghdir, and Thibault \cite{CLT 94} specifically study the convexity of  $\Phi$ in the function
space setting from the perspective of cone convexity, establishing conjugate and subdifferential formulas for  $\Phi$ under Fenchel-Rockafellar and Attouch-Br\'ezis-type qualification conditions, respectively.  
Pennanen \cite{Pen 99} studies the convexity of $\Phi$ by  developing a deep  theory of generalized differentiation  for graph-convex  mappings, which yields some results on convex  convex-composite functions as a by-product.
More recently,  based on an approach by Burachik et al. \cite{BJW 06},  
Bo\c{t}, Grad, and Wanka \cite{BGW 08, BGW 09} provided a  powerful characterization of the situation when the conjugate  and subdifferential formulae established in \cite{CLT 94}, and in our study below,  hold.

Our approach to  convex convex-composite functions differs from the above  studies:  With the exception of Pennanen \cite{Pen 99}, all the above references are set in  infinite dimension. We  work in 
finite-dimensional  space which enables us to exploit the full strength of the {\em relative topology} of convex sets. This is particularly important with respect to our main convex-analytical workhorse, which is {\em infimal convolution}, see  Section \ref{sec:InfConv} for details.  In contrast, the approach in the papers \cite{CLT 94} and  \cite{BGW 08, BGW 09}, which are  most closely related to our study, is via a  pertubation function.  Our approach via infimal convolution is  more straightforward,   hence more  accessible. In addition, at least in the finite-dimensional setting, it yields  more refined results with relaxed monotonicity and without lower semicontinuity of the functions in play.
 Our analysis is also facilitated in that we study the simple format $g\circ F$ first, and then extend it to the additive composite setting $f+g\circ F$.  Thus, a main contribution of the paper is a  simplified derivation of the central convex-analytical results for convex convex-composite functions, resulting in  refined statements about  conjugacy and subdifferentiation,  under verifiable, point-based Slater-type conditions, see e.g. Theorem \ref{th:ConjMain} and  Corollary \ref{cor:Sub} for the simple $g\circ F$ setting, and Corollary \ref{cor:AddConj} for the additive composite setting. Moreover, inspired by 
Theorem \ref{eq:cvx-comp},  we emphasize and study  the case where $F$ is convex with resepect to the negative horizon cone of $g$, see Lemma \ref{lem:HorizonIncrease} and Corollary \ref{cor:Conj2}. Given the significance of this case for the applications, see Section \ref{sec:App},  this was certainly  lacking in the above mentioned references on the topic. 
 Another major   contribution of our study is to illustrate the versatility and unifying character of the  calculus of convex convex-composite setting by a series of applications of very different flavors.   These  go beyond what was considered before, and include connections to  conic programming in Section \ref{sec:ConicProg}, and to   matrix analysis and modern matrix  optimization, see  Sections \ref{sec:Kiefer}-\ref{sec:Lewis}. More concretely,   Section \ref{sec:Kiefer} contains a new extension of the {\em matrix-fractional function} \cite{BGH17, BGH 18, BuH15} to the complex domain. Section \ref{sec:VGF} provides new,  short proofs for the  conjugate and subdifferential  of {\em variational Gram functions} \cite{BGH 18,JFX 17}, and Section \ref{sec:Lewis} gives a new  proof of Lewis'  well-known result on {\em spectral functions} \cite{Lew 95, Lew 96} for the convex case. Section 5.6 extends a Farkas-type result due to Bot et al. \cite{Bot et al. 06}.

\medskip

\noindent 
{\em Notation:} Throughout $\bE$ denotes a Euclidean space, i.e. a finite-dimensional, real vector space with inner product  denoted by $\ip{\cdot}{\cdot}$.
For $A,B\subset \bE$ their {\em Minkowski sum} is given by  $A+B:=\set{a+b}{a\in A, \; b\in B}$. In case  $A=\{x\}$ we simply write $x+B:=\{x\}+B$.  We write $\rbar:=\R\cup\{\pm \infty\}$ for the extended real line. For a complex number $z=a+bi\;(a,b\in \R)$ and its complex conjugate $\bar z:=a-bi$,  $\Real z:=\frac{z+\bar z}{2}$ is its real part.
 Given a linear map $L:\bE_1\to \bE_2$, we write  $L^*$ for  its {\em adjoint}  in the sense of linear algebra which is used accordingly for matrices. We denote the space of real  and complex   $n\times m$ matrices by $\bR^{n\times m}$ and $\bC^{n\times m}$, respectively.   The $n\times n$ identity matrix    is denoted by $I_n$.  The space of Hermitian   $n\times n$    matrices is denoted by $\bH^n$  and   the  positive semidefinite and positive definite   matrices in $\bH^n$  are given by $\bH_+^n$ and $\bH^n_{++}$, respectively. We set $\bS^n:=\bH^n\cap \R^{n\times n}$ ,  $\bS^n_+:=\bH_+^n\cap \R^{n\times n}$ and $\bS^n_{++}:=\bH^n_{++}\cap \R^{n\times n}$. The trace of $A\in \bC^{n\times n}$   is denoted by $\tr A$.

\section{Preliminaries}\label{sec:Prelim}

\noindent
In what follows $\bE$ denotes a Euclidean space, i.e. a finite-dimensional, real vector space with inner product  denoted by $\ip{\cdot}{\cdot}$.   Given $S\subset \bE$, its {\em convex hull} $\co S$ is the smallest convex set containing $S$.
Given a convex set $C\subset \bE$, its {\em affine hull}  is the smallest affine set that contains $C$ or, equivalently, 
$\aff C=\lin C+\bar x$ for any $\bar x\in C$.
  The  {\em relative interior} of $C$ is the interior in the topology relative to its affine hull, i.e.
$
\ri C:=\set{x\in C}{\exists \varepsilon>0: B_\varepsilon(x)\cap \aff C\subset C}.
$
The {\em horizon (or recession) cone} of (the convex set) $C$ is the closed, convex cone given by
\[
C^\infty:=\set{v\in \bE}{\forall t>0: \bar x+tv\in \cl C\;(\bar x\in C)}.
\] 
We call a function $f:\bE\to \rbar$ {\em convex} if  its {\em epigraph}  $\epi f:=\set{(x,\alpha)\in \bE \times \R}{f(x)\leq \alpha}$ is  a convex set. We say that $f$ is {\em proper} if its {\em domain} $\dom f:=\set{x}{f(x)\leq +\infty}$ is nonempty and $f$ does not take the value $-\infty$. We call $f$ {\em lower semicontinuous (lsc)} or {\em closed} if $\epi f$ is closed. This is equivalent to saying that $f$ equals its  {\em closure} or {\em lower semicontinuous hull} $\cl f:\bE\to\rbar$ given by
\[
(\cl f)(x):=\inf\set{\alpha}{\exists \{x_k\}\to x:f(x_k)\to \alpha}\quad (x\in \bE).
\] 
We employ the following abbreviations:
\begin{eqnarray*}
\Gamma(\bE)& := & \set{f:\bE\to\rp}{f\;\text{is proper and convex}},\\
\Gamma_0(\bE) & := & \set{f\in \Gamma(\bE)}{f\;\text{is lower semicontinuous}}.
\end{eqnarray*}
When the underlying space is clear, we   simply write $\Gamma$ and $\Gamma_0$, respectively.  A central operation that 
maps $\Gamma_0$ (one-to-one) to itself is  {\em Fenchel conjugation}: For a  function $f:\bE\to\rp$ its {\em (Fenchel) conjugate} $f^*:\bE\to \rbar$ is given by
\begin{equation}\label{eq:Conjugate}
f^*(y)=\sup_{x\in \bE}\{\ip{y}{x}-f(x)\}\quad (y\in\bE).
\end{equation}
The {\em biconjugate} of $f$ is $f^{**}:=(f^*)^*$.  It is well known, see e.g. \cite[Theorem~12.2]{Roc 70}, that for $f\in \Gamma$ we have $\cl f=f^{**}\in \Gamma_0$.  The  {\em subdifferential} of $f$ at $\bar x\in \dom f$ is the set 
\[
\p f(\bar x):=\set{v\in \bE}{f(\bar x)+\ip{v}{x-\bar x}\leq f(x)\;(x\in \bE)}.
\]
Subdifferential and conjugacy operator satisfy 
the {\em Fenchel-Young inequality} 
\begin{equation}\label{eq:FYI}
f(x)+f^*(v)\geq \ip{x}{v}\quad (x,v\in \bE),
\end{equation}
where 
\begin{equation}\label{eq:FYE}
f(x)+f^*(v) =\ip{x}{v}\iff v\in \p f(x)\quad (x,v\in \bE).
\end{equation}
For $f\in \Gamma(\bE)$, and  its {\em horizon cone} is given by 
$
\hzn f:=\lev{f}{\alpha}^\infty,
$
where $\lev{f}{\alpha}:=\set{x\in \bE}{f(x)\leq \alpha}$ is any nonempty (sub)level set of $f$ to the level $\alpha\in \R$.

\subsection{Infimal convolution}\label{sec:InfConv}

The most elementary, yet most central  convexity-preserving functional operation is addition.   
It is paired in duality,  in the sense of Fenchel conjugacy  with {\em infimal convolution}, where for two functions $f, g\colon\bE\to\rp$ their infimal convolution is defined by
\[
(f\Box g)(x) := \inf_{y\in\bE}\;\{f(y)+g(x-y)\}\quad (x\in\bE).
\]
If the infimum is attained for all $x\in \dom f\Box g=\dom f+\dom g$, we slightly abuse notation and write  
\[
(f\Box g)(x) = \min_{y\in\bE}\;\{f(y)+g(x-y)\}.
\]
The following classical result, see e.g.  \cite[Theorem 16.4]{Roc 70}, clarifies the conjugacy relation between addition and infimal convolution of convex functions.

\begin{theorem}[Conjugacy and infimal convolution]\label{th:AttBre} Let $f,g\in \Gamma(\bE)$. Then  $(\cl f+\cl g)^*=\cl(f^*\Box g^*)$.
 If 
 \[
 \ri(\dom f)\cap \ri(\dom g)\neq \emptyset
 \]
  then the closures are superfluous  and the infimum is attained for all $y\in \dom (f^*\Box g^*)$, i.e. 
 \[
  (f+g)^*(y)=\min_{u\in \bE} \{f^*(u)+g^*(y-u)\}.
 \]
\end{theorem}

\noindent
The second, most central convexity-preserving operation is (pre)-composition with a linear (affine) map: 
Given a closed, proper, convex function $g:\bE_2\to \rp$ and a  linear  map $F:\bE_1\to \bE_2$, 
it is readily seen  that   the conjugate  $(g\circ F)^*=\cl (F^*g^*)$ is the lower semicontinuous hull of the function
\[
(F^*g^*)(y)= \inf \set{g^*(z)}{F^*(z)=y}.
\]  
In order to ensure that the closure operation can be omitted, the relative interior condition 
\[
F^{-1}(\ri (\dom g))\neq \emptyset
\]
 can be used, see \cite[Theorem 16.3]{Roc 70}. This result is, in essence, equivalent to Theorem \ref{th:AttBre}. Combining the two,  one arrives at the strong duality theorem of  {\em Fenchel-Rockafellar duality} for problems of  the form 
\begin{equation}
\label{eq:p1}
\min_{x\in\bE_1}\; f(x)+g(F(x)),
\end{equation} 
Under the qualification condition
\[
0\in \ri(\dom g-F(\dom f))
\]
it can be established that 
\[
\inf_{x\in\bE_1} \; f(x)+g(F(x))=\max_{y\in\bE_2} f^*(F^*(y))+g^*(-y).
\]
Evidently, we are interested in the case  of problem \eqref{eq:p1} where $F$ is not necessarily linear, but the composition $g\circ F$ is still convex, which is exactly the convex convex-composite setting.
%

\section{$K$-convexity}\label{sec:KConv}

\noindent
A subset $K\subset \bE$ is called a {\em cone} if 
\[
\lambda x\in K \quad (\lambda\geq 0, x\in K).
\]
A cone $K$ is convex if and only if $K+K\subset K$ and it is (called) {\em pointed} if $K\cap (-K)=\{0\}$.  
The {\em polar (cone)}  of a cone $K$ is defined by $K^\circ:=\set{v\in \bE}{\ip{v}{x}\leq 0\;(x\in K)}$. 
Given a cone $K\subset \bE$  the relation 
\[
x\leq_Ky \quad :\iff\quad  y-x\in K\quad (x,y\in \bE)
\]
induces an  ordering on $\bE$ which is a partial ordering if $K$ is  convex and  pointed, see  e.g. \cite[Proposition 3.38]{RoW 98}. We attach to $\bE$ 
a {\em largest element}  $+\infty_\bullet$ with respect to said  ordering,  which satisfies 
\[
x\leq_K +\infty_\bullet\quad(x\in \bE).
\]
We will set $\bE^\bullet:=\bE\cup\{+\infty_\bullet\}$.  
For a function $F:\bE_1\to\bE_2^\bullet$ its {\em domain, graph and range} 
are defined respectively as 
\[
\begin{aligned}
\dom F &:=  \set{x\in \bE_1}{F(x)\in \bE_2}, \\
\gph F&:=\set{(x,F(x))\in\bE_1\times\bE_2}{x\in\dom F},\\
\rge F&:=\set{F(x)\in\bE_2}{x\in\dom F}.
\end{aligned}
\]
We call $F$ {\em proper} if $\dom F\neq \emptyset$.
The following concept is central to our study.
\begin{definition}[$K$-convexity]\label{def:KConv} Let $K\subset \bE_2$ be a  cone and $F:\bE_1\to \bE_2^\bullet$. Then we call $F$ {\em convex with respect to $K$} or {\em $K$-convex} if its {\em $K$-epigraph}
\[
\Epi{K}{F}:=\set{(x,v)\in \bE_1\times \bE_2}{F(x)\leq_K v}
\] 
is convex (in $\bE_1\times \bE_2$).
\end{definition}
\noindent
We point out that,  in the setting of Definition \ref{def:KConv}, a $K$-convex function $F$ has a convex domain as 
$
\dom F=L(\Epi{K}{F})
$
where $L\colon\bE_1\times\bE_2\to\bE_1\colon (x,v)\mapsto x$. Thus,  $F$ is $K$-convex if and only if 
\[
F(\lambda x+(1-\lambda)y)\leq_K \lambda F(x)+(1-\lambda)F(y)\quad (x,y\in \dom F,\;\lambda\in  [0,1]).
\]
Clearly, we can recover the traditional notion of a convex function $g:\bE\to\rp$  by using $K=\R_+$ in Definition \ref{def:KConv}. Thus the following lemma   is a generalization of the relative interior formula for the epigraph of  an ordinary convex function, see e.g.  \cite[Lemma 7.3]{Roc 70}.



\begin{lemma}[Relative interior of $K$-epigraph]\label{lem:RiKEpi} Let $K\subset \bE_2$ be a  convex cone, and let  $F:\bE_1\to\bE_2^\bullet$ be proper and $K$-convex.
 Then 
\begin{eqnarray*}
 \ri(\Epi{K}{F})& =  & \set{(x,v)}{x\in \ri(\dom F),\; F(x)\leq_{\ri(K)} v}\\
 & = &\big( \Epi{(\ri K)}{F}\big)\bigcap\big(\ri(\dom F)\times \bE_2\big). 
 \end{eqnarray*}
 In particular, if $\dom F=\bE_1$, then $\ri(\Epi{K}{F})=\Epi{(\ri K)}{F}$.
\end{lemma}

\begin{proof}
For  $x\in\bE_1$ set $
C_x=\set{v\in\bE_2}{(x,v)\in\Epi{K}{F}}.$
It follows from \cite[Theorem~6.7, Corollary~6.6.2]{Roc 70} that
\[
\ri C_x= \ri\big(L_x^{-1}(K)\big)=L_x^{-1}(\ri K)=\set{v\in\bE_2}{v\in F(x)+\ri K},
\]
where $L_x\colon v\mapsto v-F(x)$. Hence, by \cite[Theorem~6.8]{Roc 70}, we find that
\[
\begin{aligned}
(x,v)\in\ri(\Epi{K}{F})&\Longleftrightarrow x\in\ri\set{x\in\bE_1}{C_x\not=\emptyset}\;\text{and}\; v\in\ri C_x\\
&\Longleftrightarrow x\in\ri(\dom F) \;\text{and}\; v\in\ri C_x\\
&\Longleftrightarrow x\in\ri(\dom F)\;\text{and}\; v\in F(x)+\ri K. 
\end{aligned} 
\]
\end{proof}

\noindent
We now turn  our attention to the composite setting:
For $F:\bE_1\to\bE_2^\bullet$ and  $g:\bE_2\to \rp$ we define the composition of $g$ and $F$ as
\begin{equation}\label{eq:Comp}
(g\circ F)(x)=
\begin{cases}
 g(F(x)), & \text{if }  x\in \dom F,\\ +\infty, & \text{else}. \end{cases}
\end{equation}
In particular, given $v\in \bE_2$ and the linear form $\ip{v}{\cdot}:\bE_2\to\R$, we set $ \ip{v}{F}:= \ip{v}{\cdot} \circ F
$, i.e. \[
    \ip{v}{F}(x)=\begin{cases}
    \ip{v}{F(x)}, & \text{if }  x\in \dom F,\\ +\infty,  & \text{else.}\end{cases}
\]
This scalarization of $F$ is   central to our study, a fact which  is already foreshadowed in the next two  results.  These are 
in a similar form already present in the literature, e.g. \cite{Pen 99}, but we give the elementary proofs for completeness.

\begin{theorem}[Scalarization of $K$-Convexity] \label{thm:K convexity}
Let $K\subset \bE_2$ be a cone, let   $F:\bE_1\to \bE_2^\bullet$ and $v\in \bE_2$. Then the following hold: 
\begin{itemize}
\item[a)]  If $F$ is $K$-convex then $\ip{v}{F}$ is convex for all $v\in -K^\circ$. 
\item[b)] The converse  of a) holds true if $K$ is closed and convex.
\end{itemize}
Consequently, if $K$ is a closed, convex cone, then $F$ is $K$-convex
if and only if $\ip{v}{F}$ is convex for all $v\in -K^\circ$. 
\end{theorem}
\begin{proof}
a) Suppose that $F$ is $K$-convex and let $v\in -K^\circ$, $x, y\in\dom F=\dom \ip{v}{F}$, and $\lambda\in [0,1]$. Since 
$\lambda F(x)+(1-\lambda)F(y)-F(\lambda x+(1-\lambda)y)\in K$,
it follows that
\[
\ip{v}{\lambda F(x)+(1-\lambda)F(y)-F(\lambda x+(1-\lambda)y)}\geq 0,
\]
and hence
\begin{eqnarray*}
\lambda\ip{v}{F}(x)+(1-\lambda)\ip{v}{F}(y) & = & \ip{v}{\lambda F(x)+(1-\lambda)F(y)}\\
&\geq &\ip{v}{F(\lambda x+(1-\lambda)y)}\\
& = & \ip{v}{F}(\lambda x+(1-\lambda)y).
\end{eqnarray*}

\noindent
b) Let $x, y\in\bE_1$, $\lambda\in [0,1]$, and $v\in -K^\circ$. Since $\ip{v}{F}$ is convex we have
\[
\begin{aligned}
0 & \leq \lambda\ip{v}{F}(x)+(1-\lambda)\ip{v}{F}(y)- \ip{v}{F}(\lambda x+(1-\lambda)y)\\
&=\ip{v}{\lambda F(x)+(1-\lambda)F(y)-F(\lambda x+(1-\lambda)y)}\\ 
\end{aligned}
\]
and therefore
\[
\lambda F(x)+(1-\lambda)F(y)-F(\lambda x+(1-\lambda)y)\in -(-K^\circ)^\circ =K 
\]
as   $K$ is closed and convex, see \cite[Theorem~14.1]{Roc 70}.
\end{proof}

\noindent
For the  next result  recall that the {\em indicator function}  of  a  nonempty set $S\subset \bE$  is given by 
\[
\delta_S:x\in\bE\mapsto\left\{\begin{array}{rcl}
0, & \text{if} &    x\in S,
\\ +\infty, & \text{else.} 
\end{array}\right.
\]
Its conjugate is the {\em support function} $\sigma_S:\bE\to \rp$ of $S$  given by
\[
\sigma_S(y)=\delta^*_S(y)=\sup_{x\in S}\ip{x}{y}.
\]

\begin{lemma}\label{lem:Support} Let $F:\bE_1\to \bE_2^\bullet$ and let $K\subset \bE_2$ be a closed, convex cone. Then the following hold for all $(u,v)\in \bE_1\times \bE_2$:
\begin{itemize}
\item[a)]  $\sig_{\Epi{K}{F}}(u,v)=\sig_{\gph F}(u,v)+\delta_{K^\circ}(v)$.
\item[b)]  $\sigma_{\gph F}(u,-v)=\ip{v}{F}^*(u)$.
\item[c)] If $F$ is linear  then $\ip{v}{F}^*=\delta_{\{F^*(v)\}}$.
\end{itemize}
\end{lemma}
\begin{proof} a)  Observe that
\begin{eqnarray*}
\sig_{\Epi{K}{F}}(u,v)& = & \sup_{(x,y)\in \Epi{K}{F}}\ip{(u,v)}{(x,y)}\\
& = & \sup_{(x,z)\in \bE_1\times K}\ip{(u,v)}{(x,F(x)+z)}\\
& = & \sup_{x\in \bE_1}\ip{(u,v)}{(x,F(x))} +\sup_{z\in K}\ip{z}{v}\\
& = & \sig_{\gph F}(u,v)+\delta_{K^\circ}(v),
\end{eqnarray*}
where the last identity uses the well-known fact that $\sig_{K}=\delta^*_{K}=\delta_{K^\circ}$ for a closed, convex cone, see \cite[Example 11.4]{RoW 98}.
\smallskip

\noindent
b) We have 
\begin{eqnarray*}
\sig_{\gph F}(u,-v) & = & \sup_{x\in \dom F}\ip{(u,-v)}{(x,F(x))}\\
& = & \sup_{x\in \bE_1} \{\ip{u}{x}- \ip{v}{F}(x)\}\\
&= & \ip{v}{F}^*(u).
\end{eqnarray*}
\smallskip

\noindent
c) By linearity of $F$ we have 
\[
\ip{v}{F}^*(u)=\sup_{x\in \bE_1} \{\ip{u}{x}-\ip{v}{F(x)}\}=\sup_{x\in \bE_1}\{ \ip{x}{u-F^*(v)}\}=\delta_{\{F^*(v)\}}(u).
\]
\end{proof}

\noindent
The next result gives a  sufficient condition for convexity of a composite function   \eqref{eq:Comp}.

\begin{proposition}\label{prop:composite1} Let   $K\subset \bE_2$ be a  convex cone, $F:\bE_1\to \bE_2^\bullet$ $K$-convex and  $g\in \Gamma(\bE_2)$.  Then the following hold: 
\begin{itemize}
\item[a)]  $g\circ F$ is proper if and only if  $\rge F\cap \dom g\neq \emptyset$;
\item[b)] $g\circ F$ is convex if  
\begin{equation}\label{eq:FKMonotone} 
g(F(x))\leq g(y)\quad ((x,y)\in \Epi{K}{F}).
 \end{equation}
\end{itemize}

\end{proposition}
\begin{proof} a) Clear. 
\smallskip

\noindent
b) Let $v,w\in \dom F$ and $\lambda \in [0,1]$. Then 
\begin{eqnarray*}
 g(F(\lambda v+(1-\lambda)w))
& \leq & g(\lambda F( v)+(1-\lambda) F(w))\\
& \leq  &  \lambda g(F(v))+(1-\lambda) g(F(w)).
\end{eqnarray*}
Here the first inequality is due to the fact that $(\lambda v+(1-\lambda)w,\lambda F( v)+(1-\lambda) F(w))\in \Epi{K}{F} $ (as $F$ is $K$-convex) and 
assumption \eqref{eq:FKMonotone}. 
%
\end{proof}

\noindent
We point out that the property \eqref{eq:FKMonotone} used to establish convexity of the composite function $g\circ F$, and which we borrowed from \cite{Pen 99}, is clearly weaker  than saying that $g$ be {\em $K$-increasing}, i.e. 
\begin{equation}\label{eq:FKMonotone2} 
g(u)\leq g(v)\quad (u\leq_K v),
 \end{equation}
 cf. Example \ref{ex:CQandMontone} below.  However, in Lemma \ref{lem:HorizonIncrease} it will be shown that any closed, proper, convex function is  increasing with respect to its own negative horizon cone.  
  Condition \eqref{eq:FKMonotone2}   has some interesting consequences.
 
 \begin{lemma}\label{lem:KIncreas} Let $g:\bE_2\to \rp$ satisfy \eqref{eq:FKMonotone2} for some cone $K\subset \bE_2$. Then $\dom g-K=\dom g$. 
 \end{lemma}
\begin{proof} 
As $0\in K$,  we clearly have $\dom g-K\supset\dom g$.  In turn, if  $u\in\dom g$ and $v\in K$, by 
\eqref{eq:FKMonotone2}, we have $g(u-v)\leq g(u)$, which yields $u-v\in\dom g$.
\end{proof}

\noindent
We now investigate under which conditions we can also get closedness of the composite function. 
 To this end, we first make an elementary observation.

\begin{lemma}\label{lem:gConvInd} Let $K\subset \bE$ be a cone  and let  $g:\bE\to \rbar$ be $K$-increasing in the sense of \eqref{eq:FKMonotone2}. Then $g=g\Box\delta_{-K}$.
\end{lemma}

\begin{proof} For all $v\in \bE$ we have
\[
(g\Box\delta_{-K})(v)
=\inf_{u\in\bE}\{g(v-u)+\delta_{-K}(u)\}
=\inf_{u\in K}g(v+u)
=g(v),
\]
where the last equality follows from  \eqref{eq:FKMonotone2}.
\end{proof}

\noindent
The next result is a key observation which will be used multiple times in our study, but is also interesting in its own right.

\begin{lemma}\label{lem:psi}   Let   $K\subset \bE_2$ be a closed, convex cone, let $F\colon\bE_1\to \bE_2^\bullet$ be such that $\set{\ip{v}{F}}{v\in -K^\circ}\subset\Gamma_0(\bE_1)$, and let $g\in \Gamma_0(\bE_2)$ be {\em $K$-increasing}. Define $\psi:\bE_1\times \bE_2\to\rp$, 
\begin{equation} 
\label{eq:psi} \psi(p,v):=g^*(v)+\sigma_{\Epi{K}{F}}(p,-v).
\end{equation}
Then $\psi^*(x,u)=g(u+F(x))$.  
\end{lemma}
\begin{proof} We have
\[
\begin{aligned}
\psi^*(x,u)&=\sup_{(p,v)\in\bE_1\times\bE_2}\{\ip{x}{p}+\ip{u}{v}-\psi(p,v)\}\\
&=\sup_{(p,v)\in\bE_1\times\bE_2}\big\{\ip{x}{p}+\ip{u}{v}-g^*(v)-\ip{v}{F}^*(p)-\delta_{-K^\circ}(v)\big\}\\
&=\sup_{v\in\bE_2}\big\{\ip{u}{v}-g^*(v)-\delta_{-K^\circ}(v) +\sup_{p\in\bE_1}\{\ip{x}{p}-\ip{v}{F}^*(p)\}\big\}\\
&=\sup_{v\in\bE_2}\{\ip{u}{v}-g^*(v)-\delta_{-K^\circ}(v) +\ip{v}{F}(x)\}\\
&=\sup_{v\in\bE_2}\{\ip{u+F(x)}{v}-g^*(v)-\delta_{-K^\circ}(v)\}\\
&=(g^*+\delta_{-K^\circ})^*(u+F(x))\\
&=(g\Box\delta_{-K})^{**}(u+F(x))\\
&=g(u+F(x)).
\end{aligned}
\]
Here the second identity  comes from Lemma \ref{lem:Support} a), the fourth  uses the fact that  $\ip{v}{F}^{**}=\ip{v}{F}\;(v\in -K^\circ)$ by assumption. The last before last identity   employs  Theorem \ref{th:AttBre} and the fact that $(g\Box\delta_{-K})^{**}=\cl(g\Box \delta_{-K})$.   The last equality follows from  Lemma \ref{lem:gConvInd} and using  $g^{**}=g$.   
\end{proof}

\noindent
We close out this section with sufficient conditions for closedness (and convexity) of the composite functions from \eqref{eq:Comp}.

\begin{corollary}\label{cor:composite2} Under the assumptions of  Lemma \ref{lem:psi} the composite function $g\circ F$ is closed and convex. In particular, if, in addition, $\rge F\cap \dom g\neq \emptyset$, then  $g\circ F\in \Gamma_0(\bE_1)$.
\end{corollary}

\begin{proof}   By Lemma \ref{lem:psi}  we have $g\circ F=\psi^*(\cdot, 0)$, which is closed and convex.  The properness is exactly the additional assumption.
\end{proof}

\noindent
We point out that Corollary \ref{cor:composite2} generalizes the well-established scalar case as in e.g. \cite[Proposition 2.1.7, Chapter B]{HUL 01}.

\section{Convex analysis of convex convex-composite functions} \label{sec:Main}

\noindent
We commence  with the main result of this section. Throughout,  we use the standing assumption that the function $F:\bE_1\to \bE_2^\bullet$ in question is proper, i.e. $\dom F\neq \emptyset$.

\begin{theorem}[Conjugate of composite function]\label{th:ConjMain} Let  $K\subset \bE_2$ be a closed, convex cone, $F:\bE_1\to \bE_2^\bullet$ $K$-convex and $g\in \Gamma(\bE_2)$ such that \eqref{eq:FKMonotone} is satisfied.
  Then the following hold:
 \begin{itemize}
\item[a)]   $(g\circ F)^*\leq \cl \eta$, where 
\[
\eta(p)=\inf_{v\in -K^\circ} g^*(v)+\ip{v}{F}^*(p).
\]
Equality  holds if, in addition, $\ip{v}{F}\in\Gamma_0(\bE_2)$ for all $v\in-K^\circ$, $g$ is lower semicontinuous and $K$-increasing, and $\rge F\cap \dom g\not=\emptyset$.
\item [b)] If
\begin{equation}\label{eq:CQ1}
 F(\ri (\dom F))\cap \ri(\dom g-K)\neq \emptyset. 
\end{equation}
the function $\eta$ in a) is closed, proper and convex  and the infimum is a minimum (possibly $+\infty$).   
\item[c)] Under the assumptions of b)  we have 
\[
(g\circ F)^*(p)=\min_{v\in -K^\circ} g^*(v)+\ip{v}{F}^*(p)
\]
with $\dom (g\circ F)^*=\set{p\in \bE_1}{\exists v\in \dom g^*\cap (-K^\circ): \ip{v}{F}^*(p)<+\infty}$. 
\end{itemize}
\end{theorem}
\begin{proof}

a)  Define $\phi\colon\bE_1\times \bE_2\to\rp\colon (x,y)\mapsto g(y)$ and observe that 
\begin{equation}\label{eq:ConjPhi}
\phi^*(u,v)=\delta_{\{0\}}(u)+g^*(v).
\end{equation}
Hence  we find that
\begin{equation}
\label{eq:2019_05_08_1}
\begin{aligned}
(g\circ F)^*(p) & =  \sup_{x\in \bE_1} \left\{  \ip{x}{p} -g(F(x)) \right\}\\
& =  \sup_{(x,y)\in \Epi{K}{F}} \left\{ \ip{x}{p}-g(y) \right\}\\
& =  \sup_{(x,y)\in \bE_1\times \bE_2} \left\{ \ip{(p,0)}{(x,y)}-(g(y)+\delta_{\Epi{K}{F}}(x,y) \right\}\\
& = (\phi+\delta_{\Epi{K}{F}})^*(p,0) \\
&\leq   (\cl \phi+\cl\delta_{\Epi{K}{F}})^*(p,0) \\
& =  \cl\left( \phi^*\Box \sig_{\Epi{K}{F}}\right)(p,0).
\end{aligned}
\end{equation}
Here the second equality uses assumption \eqref{eq:FKMonotone}, and the inequality is due to the fact that conjugation is order-reversing.  The last identity is then due to Theorem \ref{th:AttBre}
as the functions in play are proper and convex by assumption. 
Moreover, we have
\begin{equation}
\label{eq:2019_05_14_01}
\begin{aligned}
\left(\phi^*\Box \sig_{\Epi{K}{F}}\right)(p,0)
& =  \inf_{(u,v)\in\bE_1\times \bE_2}\left\{\phi^*(u,v)+\sig_{\Epi{K}{F}}(p-u,-v)\right\}\\
& =  \inf_{v\in \bE_2} \left\{g^*(v) + \sig_{\Epi{K}{F}}(p,-v)\right\}\\
& =  \inf_{v\in -K^\circ}  \left\{ g^*(v) + \sig_{\gph F}(p,-v)\right\}\\
& =  \inf_{v\in -K^{\circ}} \left\{g^*(v)+ \ip{v}{F}^*(p) \right\},
\end{aligned}
\end{equation}
where the second identity uses \eqref{eq:ConjPhi} and the third and fourth  rely on Lemma~\ref{lem:Support}. 
This shows the first desired statement. 

To show that equality holds under the additional assumptions  observe that with the function $\psi$ in \eqref{eq:psi}, we have $\eta=\inf_{v}\psi(\cdot,v )$. Hence $\eta$ is convex and, by Lemma \ref{lem:psi}, it holds that
$
\eta^*=\psi^*(\cdot,0)=g\circ F.
$
As $g\circ F$ is closed, proper, convex  by Corollary \ref{cor:composite2}, $\eta$ is convex and proper, and so
$
\cl \eta =\eta^{**}=(g\circ F)^*,
$
which is the desired equality.

\smallskip

\noindent
b) This follows from Theorem \ref{th:AttBre}   while observing that 
\begin{eqnarray*}
\ri(\dom \phi)\cap \ri(\dom \delta_{\Epi{K}{F}})\neq \emptyset & \Longleftrightarrow &  \bE_1\times\ri(\dom g)\cap \ri(\Epi{K}{F})\neq \emptyset\\
& \Longleftrightarrow & \exists x\in \ri(\dom F): \; F(x)\in \ri(\dom g)-\ri K\\
& \Longleftrightarrow & F(\ri (\dom F))\cap \ri(\dom g-K)\neq \emptyset. 
\end{eqnarray*}
Here the second equivalence relies on  Lemma \ref{lem:RiKEpi}.
\smallskip

\noindent
c) The first statement follows from a) and b) and Theorem \ref{th:AttBre}. The expression of $\dom (g\circ F)^*$  is an immediate consequence of that.
\end{proof}

\noindent
We would like to point out that the technique of proof  based on infimal convolution  is an extension of the approach employed by Hiriart-Urruty in \cite{HiH 06}.  However, our setting is much more general, and  our point-based, Slater-type  qualification condition \eqref{eq:CQ1} differs substantially from the Fenchel-Rockafellar-type condition used in \cite{HiH 06}.

We now  continue with a whole sequence of rather immediate consequences of Theorem \ref{th:ConjMain}. 
The first one is  a subdifferential formula. 
To this end, we first establish an auxiliary result.
 
 \begin{lemma}\label{lem:SubAux}  Let $K\subset \bE_2$ be a closed, convex cone. Then the following hold:
 \begin{itemize}
 \item[a)]    Let $g:\bE_2\to \rp$ such that   \eqref{eq:FKMonotone2} holds  and let $\bar x\in F^{-1}(\dom g)$.  Then 
 $
 \p g(F(\bar x))\subset -K^\circ.
 $
 \item[b)] Under the assumptions of Theorem \ref{th:ConjMain} c) let $\bar y\in \p(g\circ F)(\bar x)$ and $\bar v\in\argmin_{v\in -K^\circ} \left\{g^*(v)+\ip{v}{F}^*(\bar y)\right\}$. Then $\bar v\in \p g(F(\bar x))$.

 \end{itemize}

\end{lemma} 
 \begin{proof} a)  Let $\bar v\in \p g(F(\bar x))$, i.e.
 \[
 g(F(\bar x))+\ip{\bar v}{y-F(\bar x)}\leq g(y)\quad (y\in \bE_2).
 \]
 In particular, letting $y=F(\bar x)-k\;(k\in K)$ and by \eqref{eq:FKMonotone2}, this implies
 \[
0\leq  g(F(\bar x))-g(F(\bar x)-k)\leq \ip{\bar v}{k}\quad (k\in K).
 \]
 Therefore, $\bar v\in -K^\circ$.
 \smallskip
 
 \noindent 
 b)  We observe that
\begin{eqnarray*}
\ip{\bar x}{\bar y} & =  & (g\circ F)(\bar x) +(g\circ F)^*(\bar y)\\
& = & g(F(\bar x))+g^*(\bar v)+\sigma_{\gph F}(\bar y,-\bar v)\\
& = &  g(F(\bar x))+g^*(\bar v)+\sup_{z\in\bE_1}\ip{z}{\bar y}-\ip{F(z)}{\bar v}.
\end{eqnarray*}
Here the first identity is due to \eqref{eq:FYE}, and  the second is due to Theorem \ref{th:ConjMain}  c) in combination with 
Lemma \ref{lem:Support}. Now insert  $z:=\bar x$ to obtain
\begin{eqnarray*}
& & \ip{\bar x}{\bar y}   \geq g(F(\bar x)) + g^*(\bar v) +\ip{\bar x}{\bar y}-\ip{F(\bar x)}{\bar v} \\
& \Longleftrightarrow& \ip{F(\bar x)}{\bar v}\geq g(F(\bar x))+g^*(\bar v)\\
& \Longleftrightarrow& \bar v\in \p g(F(\bar x)).
\end{eqnarray*}
 \end{proof}

%
%

\begin{corollary}[Subdifferential of composite functions]\label{cor:Sub}  For $g:\bE_2\to \rp$ and $F:\bE_1\to\bE_2^\bullet$ we have 
\[
\p (g\circ F)(\bar x)\supset\bigcup_{v\in \p g(F(\bar x))} \p\ip{v}{F} (\bar x)\quad (\bar x\in \dom g\circ F).
\]
Equality holds under  the assumptions of Theorem \ref{th:ConjMain} b).
\end{corollary}
\begin{proof} Let $\bar y\in \bigcup_{v\in \p g(F(\bar x))} \p\ip{v}{F} (\bar x)$, i.e. there exists 
$\bar v\in \p g(F(\bar x))$ such that $\bar y\in \p\ip{\bar v}{F} (\bar x)$. Therefore
\begin{equation}\label{eq:SGg}
g(F(\bar x))+\ip{\bar v}{y-F(\bar x)}\leq g(y)\quad (y\in \bE_2)
\end{equation}
and 
\begin{equation}\label{eq:SGgF}
\ip{\bar v}{F(\bar x)}+\ip{\bar y}{x-\bar x}\leq \ip{\bar v}{F(x)}\quad (x\in \bE_1).
\end{equation}
Inserting $y:=F(x)\;(x\in \bE_1)$ in \eqref{eq:SGg} yields 
\[
g(F(\bar x))+\ip{\bar v}{F(x)-F(\bar x)}\leq g(F(x))\quad (x\in \bE_1).
\]
Combining with  \eqref{eq:SGgF} now gives 
\[
g(F(\bar x))+\ip{\bar y}{x-\bar x}\leq g(F(x))\quad (x\in \bE_1),
\]
which shows $\bar y\in \p (g\circ F)(\bar x)$.

 To see the reverse inclusion under the additional assumptions let $\bar y\in \p(g\circ F)(\bar x)$, i.e., by \eqref{eq:FYE}, we have 
\begin{equation} \label{eq:FYComp}
\ip{\bar x}{\bar y}  =  (g\circ F)(\bar x)+(g\circ F)^*(\bar y).
\end{equation}
Now let $\bar v\in\argmin_{v\in -K^\circ} \left\{g^*(v)+\ip{v}{F}^*(\bar y)\right\}$. Then from \eqref{eq:FYComp} we infer
\begin{eqnarray*}
\ip{\bar x}{\bar y} & = & g(F(\bar x))+g^*(\bar v)+\ip{\bar v}{F}^*(\bar y)\\
& = & \ip{F(\bar x)}{\bar v}+\ip{\bar v}{F}^*(\bar y)\\
& = & \ip{\bar v}{F}(\bar x)+\ip{\bar v}{F}^*(\bar y),
\end{eqnarray*}
where the first identity is due to  the choice of $\bar v$ and the second is \eqref{eq:FYE} with Lemma \ref{lem:SubAux} b).  This yields $\bar y\in \p\ip{\bar v}{F}(\bar x)$. 
%
\end{proof}

\begin{remark}\label{rem:DiffScalar}  Given $F:\bE_1\to \bE_2$ $K$-convex   and   $v\in -K^\circ$, the convex    function $\ip{v}{F}$ enjoys a rich (sub)differential calculus under additional assumptions on the continuity or differentiability of $F$. For instance, if $F$ is locally Lipschitz,  then so is $\ip{v}{F}$ with 
\[
\p \ip{v}{F}(\bar x)= \bar \p F(\bar x)^*v,
\]
where $\bar \p F(\bar x)$ is  {\em  Clarke's generalized Jacobian} of $F$ at $\bar x$, see \cite{Cla 83, RoW 98}. In particular, if $F$ is differentiable, then $\ip{v}{F}$ is continuously differentiable with 
\[
\nabla \ip{v}{F}(\bar x) = \nabla F(\bar x)^*v.
\]
\end{remark}

\noindent
The next result  shows that our setting $g\circ F$ in fact covers the seemingly  more general setting  from \eqref{eq:cvx-comp} with an additional additive term. Once more, the proof relies essentially  on infimal convolution.

\begin{corollary}[Conjugate and subdifferential  of additive composite functions]\label{cor:AddConj} Under the assumptions of Theorem \ref{th:ConjMain} let  \eqref{eq:FKMonotone}  hold. In addition,  let $f\in \Gamma$ and consider the qualification condition 
\begin{equation}\label{eq:CQ2}
F(\ri (\dom f) \cap  \ri(\dom F))\cap \ri (\dom g-K)\neq \emptyset.
\end{equation}
Then the following hold:
\begin{itemize}
\item[a)] If \eqref{eq:CQ2} holds then  
\[
(f+g\circ F)^*(p)=\min\limits_{\overset{v\in -K^\circ,}{y\in \bE_1}} g^*(v)+f^*(y)+\ip{v}{F}^*(p-y),
\]
\item[b)] We  have 
  \[
  \p (f+g\circ F)(\bar x)\supset \p f(\bar x)+\bigcup_{v\in \p g(F(\bar x))} \p\ip{v}{F}(\bar x) \quad(\bar x\in \dom f\cap \dom g\circ F).
  \]
  Equality holds under \eqref{eq:CQ2}.
\end{itemize}

\end{corollary}
\begin{proof} a)  Define $\tilde g:\R\times \bE_2\to \rp$ by $\tilde g(s,y):=s+g(y)$.  Then $\tilde g\in \Gamma_0(\R\times \bE_2)$ with 
\begin{equation}\label{eq:ConjTild}
\tilde g^*(t,v)=\delta_{\{1\}}(t)+g^*(v).
\end{equation}
Moreover, define  $\tilde F:\bE_1\to \rp\times \bE_2^\bullet$ by $\tilde F(x)=(f(x),F(x))$. Then $\dom \tilde F=\dom f\cap\dom F$.    Setting $\tilde K :=\R_+
 \times K$,  we find that $\tilde F$ is $\tilde K$-convex, $\tilde g\circ \tilde F=f+g\circ F\in \Gamma_0(\bE_1)$ and $\tilde g\circ \tilde F$ satisfies \eqref{eq:FKMonotone} with $\tilde K$  (since $g\circ F$ satisfies it with $K$). Moreover, as $\dom \tilde g=\R\times \dom g$,  we realize that condition \eqref{eq:CQ1} for   $\tilde g,\tilde F$ and $\tilde K$ amounts to \eqref{eq:CQ2}. Therefore, we obtain
 \begin{eqnarray}
 \label{eq:conj_comp}
 (f+g\circ F)^*(p) & = & (\tilde g\circ \tilde F)(p)\notag\\
 & = & \min_{(t,v)\in -\tilde K^\circ} \tilde g^*(t,v)+\ip{(t,v)}{\tilde F}^*(p)\notag\\
 & = & \min_{v\in -K^\circ} g^*(v)+\sup_{x\in \bE_1}\{\ip{p}{v}-(f(x)+\ip{v}{F}(x))\}\\
 & = & \min_{v\in -K^\circ} g^*(v)+ (f+\ip{v}{F})^*(p)\notag\\
 & = &  \min_{v\in -K^\circ} g^*(v)+(f^*\Box \ip{v}{F}^*)(p)\notag\\
 &= & \min_{\overset{v\in -K^\circ,}{y\in \bE_1}} g^*(v)+f^*(y)+\ip{v}{F}^*(p-y).\notag
 \end{eqnarray}

\noindent
Here the second identity is due to Theorem \ref{th:ConjMain}, the third one uses \eqref{eq:ConjTild}, while the fifth relies once more on Theorem \ref{th:AttBre}, realizing that $\dom \ip{v}{F}=\dom F$ and \eqref{eq:CQ2} implies that $\ri(\dom f)\cap \ri(\dom F)\neq \emptyset$.  
\smallskip

\noindent 
b)  The first statement (without qualification condition) follows from Corollary \ref{cor:Sub} and the subdifferential sum rule \cite[Theorem 23.8]{Roc 70}. 

In order to prove the converse inclusion under \eqref{eq:CQ2}, let  $\bar y\in\partial(f+g\circ F)(\bar x)$. Hence by \eqref{eq:FYE} and \eqref{eq:conj_comp} we have
\[
\ip{\bar x}{\bar y}-(f(x)+g(F(\bar x))=(f+g\circ F)^*(\bar y) = \min_{v\in -K^\circ}g^*(v)+(f+\ip{v}{F})^*(\bar y).
\]
Hence there exists  $\bar v\in -K^\circ$ such that
\begin{equation}
\label{eq:co2}
(f+g\circ F)^*(\bar y)= g^*(\bar v)+(f+\ip{\bar v}{F})^*(\bar y).
\end{equation}
Therefore, we have
\begin{equation}
\label{eq:co1}
\begin{aligned}
\ip{\bar x}{\bar y} =f(\bar x)+g(F(\bar x))+g^*(\bar v)+(f+\ip{\bar v}{F})^*(\bar y).
\end{aligned}
\end{equation}
On the one hand, \eqref{eq:co1} and the Fenchel-Young inequality \eqref{eq:FYI} applied to $g$ imply \[
\ip{\bar x}{\bar y}\geq f(\bar x)+\ip{\bar v}{F(\bar x)}+(f+\ip{\bar v}{F})^*(\bar y),
\]
which, by \eqref{eq:FYE} applied to $f+\ip{\bar v}{F}$, yields $\bar y\in\partial(f+\ip{\bar v}{F})(\bar x)$. On the other hand, we deduce from \eqref{eq:co1} and the definition of the conjugate that
\[
\ip{\bar x}{\bar y}\geq f(\bar x) + g(F(\bar x))+g^*(v)+\ip{\bar x}{\bar y}-f(\bar x)-\ip{\bar v}{F(\bar x)},
\]
which, in turn, yields
\[
\ip{\bar v}{F(\bar x)}\geq g(F(\bar x))+g^*(\bar v),
\]
and hence, $\bar v\in\partial g(F(\bar x))$, by \eqref{eq:FYE}. Altogether, we deduce that
\[
\begin{aligned}
\partial(f+g\circ F)(\bar x)&\subset\bigcup_{v\in\partial g(F(\bar x))}\partial(f+\ip{v}{F})(\bar x)\\
&=\bigcup_{v\in\partial g(F(\bar x))}\big(\partial f(\bar x)+\partial\ip{v}{F}(\bar x)\big)\\
&=\partial f(\bar x)+\bigcup_{v\in\partial g(F(\bar x))}\partial\ip{v}{F}(\bar x)\\
&\subset \partial f(\bar x)+\partial(g\circ F)(\bar x),
\end{aligned}
\]
and hence, the conclusion follows. Here, the first equality follows from the subdifferential sum rule, see e.g. \cite[Theorem 23.8]{Roc 70}  because $\ri(\dom f)\cap\ri(\dom F)\neq\emptyset$ by \eqref{eq:CQ2},  and the last identity follows from Corollary \ref{cor:Sub}.  This establishes the desired inclusion and thus concludes the proof.
\end{proof}

Of course, using  Fermat's rule, i.e. $\argmin \Phi=(\p \Phi)^{-1}(0)$,  and the fact that $\inf \Phi=-\inf \Phi^*(0)$ for any convex function $\Phi$,  Corollary \ref{cor:AddConj}  can be exploited to derive (necessary and sufficient) optimality conditions and a duality framework for the convex-composite optimization problem \eqref{eq:cvx-comp}.  We will do this explicitly in Section \ref{sec:ConicProg} for conic programming.
General optimality conditions for the convex-composite problem \eqref{eq:cvx-comp} are derived in  \cite[Proposition 8.1 (ii)]{Pen 99} under    \eqref{eq:CQ2}.

The same expressions for the conjugate and subdifferential as in Corollary \ref{cor:AddConj} were derived   by Combari et al. \cite[Proposition 4.11]{CLT 94} and Bot et al.  \cite[Theorem~3.3, Corollary 3.4]{BGW 09}.  The arguments (adaptable to the infinite-dimensional setting) used in \cite{BGW 09, CLT 94} are based on investigating properties of a suitable perturbation function and  hence, they require different notions of  lower semicontinuity of $F$ even for the finite-dimensional setting. Our arguments, on the other hand, are in essence a synthesis of the infimal convolution approach to convex-composite functions inspired by  \cite{HiH 06} and the  $K$-convexity study in \cite{Pen 99}. This allows us to eliminate the assumptions on  lower semicontinuity of $f$, $g$ and $F$ and to relax the $K$-increasing property \eqref{eq:FKMonotone2}  of $g$ to the weaker  condition \eqref{eq:FKMonotone}. 

The results by Bot et al.  \cite[Theorem~3.3, Corollary 3.4]{BGW 09}  show that for  $f\in\Gamma_0(\bE_1)$, $g\in\Gamma_0(\bE_2)$ $K$-increasing, and  $\set{\ip{v}{F}}{v\in -K^\circ}\subset\Gamma_0(\bE_1)$, the conjugate formula from Corollary \ref{cor:AddConj} a) holds if and only if 
\begin{equation}
\label{eq:regl1}
\bigcup_{v\in\dom g^*}\big((0, g^*(v))+\epi(f+\ip{v}{F})^*\big)
\end{equation}
is closed. The subdifferential formula from Corollary \ref{cor:AddConj} b) then follows. We point out that, although  \cite[Theorem~3.3]{BGW 09} is a deep result,  condition \eqref{eq:regl1} is, in general,  not verifiable.  On the other hand, in the infinite dimensional setting its verifiability  is comparable to the one for Attouch-Br\'ezis-type conditions.

The study  by Combari et al. \cite{CLT 94}, like ours, provides verifiable conditions for the conjugate and subdifferential formulae. 
 Therefore we would briefly like to discuss the differences in assumptions, in particular the qualification conditions, when compared in the finite-dimensional setting:  In \cite{CLT 94}, the authors  assume, as we do,  the properness and $K$-convexity of all functions in play. However, they also need the assumption that   $F$ is  lower semicontinuous in a generalized sense, see \cite[Definition~3.1]{CLT 94}, and they  also assume  $g$ to be   $K$-increasing \eqref{eq:FKMonotone2}, whereas we only require the weaker assumption \eqref{eq:FKMonotone}.  We note that, under \eqref{eq:FKMonotone2}, our  Slater-type  condition  \eqref{eq:CQ2} becomes
\begin{equation}
\label{eq:CQ3c}
F(\ri (\dom f)\cap  \ri(\dom F))\cap \ri (\dom g)\neq \emptyset.
\end{equation}
The (Attouch-Br\'ezis-type) qualification condition  comparable to ours, used in \cite{CLT 94} reads
\begin{equation}
\label{eq:CQcombari}
\R_+(\dom g- F(\dom F\cap\dom f))\;\text{ is a  subspace}
\end{equation}
which, using the fact that $\dom g=\dom g-K$ and $K+F(\dom F\cap\dom f)$ is convex, yields
\[
0\in  \ri(K+F(\dom f \cap  \dom F)- \dom g)
\]
or, equivalently,
\[
\ri(K+F(\dom f \cap  \dom F))\cap\ri(\dom g)\neq\emptyset.
\]
This condition differs from \eqref{eq:CQ3c} in general.  When  $f$ and $F$ are finite-valued and $\rge F$ is convex then \eqref{eq:CQcombari} becomes
\begin{equation}
\label{eq:CQcom}
\ri(\dom g)\cap\ri(\rge F)\neq\emptyset,
\end{equation}
which is stronger than our condition, which then reads
\[
\ri(\dom g)\cap\rge  F\neq\emptyset.
\]
The formula for $(g\circ F)^*$ is also obtained in \cite[Section~4]{BGW 07} under the following qualification condition
\[
\ri(\rge F\cap\dom g+K)\cap\ri(\dom g)\neq\emptyset,
\]
which, using the fact that $\dom g=\dom g-K$, is exactly equivalent to \eqref{eq:CQcom} if $\rge F$ is convex.
This qualification condition also appeared in \cite[Theorem~2.8]{Zal 83} where the author derived a duality result as well as the formula for $\varepsilon$-subdifferential for $g\circ F$.

The following  example, and also Example \ref{ex:Spec},  illustrates that \eqref{eq:FKMonotone} is, indeed, weaker  than \eqref{eq:FKMonotone2},  and that our qualification condition \eqref{eq:CQ2} differs from, and is often weaker than \eqref{eq:CQcombari}.

\begin{example}\label{ex:CQandMontone} Let $\bE_1=\bE_2=\R^2$, $K=\R_+\times\{0\}$,  $F:(x_1, x_2)\in \bE_1\mapsto (0,x_2)$, and define $g:=\delta_{\R_+\times \R}$.
Since $
\Epi{K}{F}=\set{(x_1, x_2, v_1, v_2)\in\R^4}{v_1\geq 0, v_2=x_2}
$
is convex, $F$ is $K$-convex. Now,  if $\rge F\ni (u_1, u_2)\leq_K(v_1, v_2)$, i.e.,
$
0=u_1\leq v_1,\quad u_2=v_2,
$
then $g(u_1,u_2)=g(v_1,v_2)$. This means that $g$ satisfies \eqref{eq:FKMonotone} but not \eqref{eq:FKMonotone2} by taking $u_1<0$ and $v_1\geq 0$. 
Furthermore, we note that
$
\R_+(\dom g-\rge F)=\R_+\times \R
$
is not a  subspace, so \eqref{eq:CQcombari} is violated,  but \eqref{eq:CQ2} is satisfied as 
$
\ri(\dom g-K)\cap\rge F= \R^2\cap \{0\}\times \R  \neq\emptyset.
$

\hfill $\diamond$
\end{example}

\noindent
The next corollary follows simply from the fact that condition \eqref{eq:FKMonotone} is trivially satisfied if $g$ is $K$-increasing and that \eqref{eq:CQ1} then simplifies to 
\begin{equation}\label{eq:CQ1simple}
F(\ri ( \dom F))\cap \ri (\dom g)\neq \emptyset.
\end{equation}

\begin{corollary}\label{cor:Conj1}   Let   $K\subset \bE_2$ be a closed, convex cone, $F:\bE_1\to \bE_2^\bullet$ $K$-convex,    $g\in \Gamma$ and    $K$-increasing such that  \eqref{eq:CQ1simple} holds  (which is true in particular when $g$ is finite-valued or  $F$ is surjective and $\dom F=\bE_2$). Then 
\[
(g\circ F)^*(p)=\min_{v\in -K^\circ} g^*(v)+\ip{v}{F}^*(p)
\]
and 
\[
\p (g\circ F)(\bar x)=\bigcup_{v\in \p g(F(\bar x))} \p\ip{v}{F}(\bar x)\quad(\bar x\in\dom g\circ F).
\]
\end{corollary}
\begin{proof} Since the fact that $g$ is $K$-increasing implies \eqref{eq:FKMonotone} the assertion follows from Theorem \ref{th:ConjMain} and Corollary \ref{cor:Sub}, respectively.
\end{proof}

\noindent
The next result shows that any closed, proper, convex function $g$ is increasing with respect to $-\hzn g$.

\begin{lemma}\label{lem:HorizonIncrease} Let $g\in \Gamma_0$. Then $g$ is  $(-\hzn g)$-increasing. 
\end{lemma}
\begin{proof} Put $K:=-\hzn g$. Then \cite[Theorem~14.2]{Roc 70} yields
$
K=-(\overline{\cone}(\dom g^*))^\circ.
$
Now let $x,y\in \bE$ such that $x\preceq_K y$, i.e. $y=x+b$ for some $b\in K$. Since $g=g^{**}$ we hence find 
\begin{eqnarray*}
g(x) & =  & \sup_{z\in\dom g^*} \{\ip{x}{z}-g^*(z)\}\\
& = & \sup_{z\in \dom g^*} \{\ip{y}{z}-\ip{b}{z}-g^*(z)\}\\
& \leq  & \sup_{z\in \dom g^*} \{\ip{y}{z}-g^*(z)\}\\
& = & g(y),
\end{eqnarray*}
where the inequality relies on the fact that $\ip{b}{z} \geq 0$.
\end{proof}

\noindent
The observation from Lemma \ref{lem:HorizonIncrease} that any closed, proper, convex function is increasing with respect to its own negative horizon cone brings us back full circle  to the  starting point of our study  in Theorem \ref{thm:cvx cvxcomp}. The latter theorem is now merely an immediate consequence of Lemma \ref{lem:HorizonIncrease} and  Proposition \ref{prop:composite1}, and we observe that $F$ can even be extended-valued.   
We pursue this  setting  where $g$ is closed, proper, convex and $F$ is convex with respect to the negative horizon cone of $g$ in the next result.

\begin{corollary}\label{cor:Conj2} Let $g\in \Gamma_0(\bE_2)$ and let $F:\bE_1\to \bE_2^\bullet$ be $(-\hzn g)$-convex such that 
\[
F(\ri(\dom F))\cap \ri(\dom g)\neq \emptyset.
\]
Then 
\[
(g\circ F)^*(p)=\min_{v\in \bE_2} g^*(v)+\ip{v}{F}^*(p)
\]
and 
\[
\p (g\circ F)(\bar x)=\bigcup_{v\in \p g(F(\bar x))} \p\ip{v}{F}(\bar x) \quad(\bar x\in \dom g\circ F).
\]
\end{corollary}
\begin{proof} This follows from combining Corollary \ref{cor:Conj1} (with $K=-\hzn g$) and Lemma \ref{lem:HorizonIncrease} while observing that $ \dom g^*\subset \overline \cone(\dom g^*)=(\hzn g)^\circ$, cf.  \cite[Theorem~14.2]{Roc 70}.
\end{proof}

\noindent
The next corollary covers the  setting considered in \cite{HiH 06} where $F$ is a component-wise convex function, and whose technique of proof served as a source of inspiration for our study.

\begin{corollary} Let $g\in \Gamma(\R^m)$ be $\R^m_+$-increasing, $F:\bE\to (\R^m)^\bullet$ with  $F_i\in \Gamma(\bE)\;(i=1,\dots,m)$ such that 
\[
F\left(\bigcap_i^m{\ri(\dom F_i)}\right)\cap \ri (\dom g)\neq \emptyset.
\]
Then 
\[
(g\circ F)^*(p)=\min_{v\geq 0} g^*(v)+\left(\sum_{i=1}^mv_i F_i\right)^*(p)
\]
and 
\[
\p (g\circ F)(\bar x)=\bigcup_{v\in \p g(F(\bar x))} \sum_{i=1}^{m} v_i\p F_i(\bar x)\quad(\bar x\in  \dom g\circ F).
\]
\end{corollary}
\begin{proof} Apply Corollary  \ref{cor:Conj1} with $K=\R^m_+$ and observe that $\p g(F(\bar x))\subset \R_+^m$, cf. Lemma \ref{lem:SubAux} a), hence $\p\ip{v}{F}(\bar x)=\sum_{i=1}^mv_i \p F_i(\bar x)$ for all $v\in g(F(\bar x)).$
\end{proof}

\noindent 
Finally, as another immediate consequence of our study, we get the well-known result for the case when $F$ is linear. 

\begin{corollary}[The linear case]\label{eq:FLinear} Let $g\in \Gamma(\bE_2)$ and $F:\bE_1\to \bE_2$  linear such that 
$\rge F\cap \ri(\dom g)\neq \emptyset$. Then 
\[
(g\circ F)^*(p)=\min_{v\in\bE_2} \set{g^*(v)}{F^*(v)=p}
\]
with $\dom (g\circ F)=(F^*)^{-1}(\dom g^*)$, and 
\[
\p(g\circ F)(\bar x)=F^*\left(\p g(F(\bar x))\right)\quad (\bar x\in \dom g\circ F).
\]

\end{corollary}
\begin{proof} We notice that $F$ is $\{0\}$-convex. Hence we can apply Corollary  \ref{cor:Conj1} with $K=\{0\}$. Condition \eqref{eq:CQ1simple} then reads $\rge F\cap \ri(\dom g)\neq \emptyset$, which is our assumption. Hence we obtain
\[
(g\circ F)^*(p)=\min_{v\in -K^\circ} g^*(v)+\ip{v}{F}^*{p}=\min_{v\in \bE_2} g^*(v)+\delta_{F^*(v)}(p).
\]
where the second identity uses Lemma \ref{lem:Support}.  Moreover, $\p\ip{v}{F}(\bar x)=F^*(v)$ for all $v\in \bE_2$, which proves the subdifferential formula.
 \end{proof}

\section{Applications}\label{sec:App}

\noindent
In this section we present an  eclectic series of   applications of our study in Section \ref{sec:KConv} and  \ref{sec:Main} to illustrate the versatility of our findings and to establish connections to different areas of convex analysis, optimization and matrix analysis.

\subsection{Conic programming} \label{sec:ConicProg}

We consider the general  conic program \cite{ShS 03}
\begin{equation}\label{eq:CP}
\min f(x)\st F(x)\in -K 
\end{equation}
where $f\in \Gamma(\bE_1)$, $F:\bE_1\to \bE_2$ is $K$-convex and $K\subset \bE_2$ is a closed, convex cone.    Clearly, \eqref{eq:CP} can be written in the additive composite form 
\begin{equation}\label{eq:PFen}
\min_{x\in\bE_1} f(x)+(\delta_{-K}\circ F)(x).
\end{equation}
This fits the additive composite setting of  Corollary \ref{cor:AddConj} with $g=\delta_{-K}$ which is $K$-increasing (in fact, let $u\leq_Kv$, if $v\in -K$ then $u=u-v+v\in -K-K=-K$ and hence, in all cases, $g(u)\leq g(v)$). Moreover, the qualification condition \eqref{eq:CQ2}  reads
\begin{equation}\label{eq:Slater}
F(\ri(\dom f))\cap \ri(-K)\neq\emptyset.
\end{equation}
The Fenchel(-Rockafellar) dual problem associated with \eqref{eq:CP} via \eqref{eq:PFen} is 
\[
\max_{y\in \bE_1} -f^*(y)-(\delta_{-K}\circ F)^*(-y),
\]
while the Lagrangian dual is 
\[
\max_{v\in -K^\circ} \inf_{x\in \bE_1} f(x)+\ip{v}{F(x)}.
\]
We obtain the following duality result. 
\begin{theorem}[Strong duality and dual attainment for conic programming]\label{th:CP} Let $f\in \Gamma(\bE_1)$, $K\subset \bE_2$ a closed, convex cone, and  let $F:\bE_1\to \bE_2$ be $K$-convex. If \eqref{eq:Slater} holds then 
\[
\begin{aligned}
\inf_{x\in \bE_1} f(x)+(\delta_{-K}\circ F)(x) &= \max_{v\in -K^{\circ}}  -f^*(y)-(\delta_{-K}\circ F)^*(-y)\\
&=\max_{v\in -K^\circ} \inf_{x\in \bE_1}  f(x)+\ip{v}{F(x)}.
\end{aligned}
\]
\end{theorem}
\begin{proof}
Observe that, by Corollary \ref{cor:AddConj}, we have 
\begin{eqnarray*}
(f+\delta_{-K}\circ F)^*(0)& = & \min_{\stackrel{v\in -K^\circ,}{y\in \bE_1}}f^*(y)+\ip{v}{F}^*(-y)\\
& = & \min_{v\in -K^\circ} (f+\ip{v}{F})^*(0)\\
& = & \min_{v\in -K^\circ} \sup_{x\in \bE_1} \{-f(x)-\ip{v}{F(x)}\}.
\end{eqnarray*}
Moreover, by Theorem \ref{th:ConjMain} we have 
\[
(\delta_{-K}\circ F)^*(-y)=\min_{v\in -K^\circ} \ip{v}{F}^*(-y).
\]
Finally, since $\inf_{\bE_1} f+\delta_{-K}\circ F=-(f+\delta_{-K}\circ F)^*(0)$, this shows everything.
\end{proof}

\noindent
Theorem \ref{th:CP} furnishes  various  facts:  It shows strong duality and dual attainment for  both the Fenchel duality   (first identity) and Lagrangian duality (second identity) scheme under a generalized Slater condition. Moreover, it shows the equivalence of both duality concepts. Of course, these are well-known results, but the proof based on Corollary \ref{cor:AddConj} and  Theorem \ref{th:ConjMain}, respectively,  unifies this in a very elegant and convenient way.

Optimality conditions are also easily derived. Here observe that for a convex set $C\subset \bE$, the {\em normal cone} of $C$ at $\bar x\in C$ is 
\[
N_C(\bar x)=\p \delta_C(\bar x)=\set{v}{\ip{v}{x-\bar x}\leq 0\;(x\in C)}.
\] 

\begin{theorem}\label{th:OptCP} Let $f\in \Gamma(\bE_1)$, $K\subset \bE_2$ a closed, convex cone, and  let $F:\bE_1\to \bE_2$ be $K$-convex.  Then the condition 
\begin{equation} \label{eq:OptCP}0\in \p f(\bar x) +\bigcup_{v\in N_{-K}(F(\bar x))}\p\ip{v}{F}(\bar x)
\end{equation}
is sufficient  for $\bar x$ to be a  minimizer of \eqref{eq:CP}. Under \eqref{eq:Slater} it is also necessary. 
\end{theorem}

\begin{proof}  From  Corollary \ref{cor:AddConj} b) with $g=\delta_{-K}$,  \eqref{eq:OptCP}  it follows that $0\in \p (f+g\circ F)(\bar x)$, so $\bar x$ is a minimizer of $f+g\circ F$. 
Under   \eqref{eq:Slater}, again  by Corollary \ref{cor:AddConj} b),  $\p (f+g\circ F)(\bar x)$ equals the set on the right-hand side of \eqref{eq:OptCP}, which concludes the proof.
\end{proof}

\noindent
The differentiable case merits its own statement. 

\begin{corollary}\label{cor:OptCP}  
 Let $f:\bE_1\to \R$ be differentiable and  convex, $K\subset \bE_2$ a  closed, convex cone, and  let $F:\bE_1\to \bE_2$ be differentiable and  $K$-convex.  Then the condition 
\begin{equation}\label{eq:OptCPDiff} 
-\nabla f(\bar x)\in F'(\bar x)^*N_{-K}(F(\bar x))
\end{equation}
is sufficient  for $\bar x$ to be a  minimizer of \eqref{eq:CP}. Under the condition
\begin{equation}\label{eq:Slater2}
\rge F\cap \ri(-K)\neq \emptyset
\end{equation}
it is also necessary. 
\end{corollary}
\begin{proof} This follows immediately from Theorem \ref{th:OptCP}, Remark \ref{rem:DiffScalar} and the fact that  \eqref{eq:Slater} reduces to  \eqref{eq:Slater2} when $\dom f=\bE_1$.
\end{proof}

\subsection{The pointwise maximum of convex functions} 

In what follows we denote the unit simplex in $\R^m$ by $\Delta_m$, i.e.
$$
\Delta_m=\set{v\in\R^m}{v_i\geq 0, \; \sum_{i=1}^mv_i=1}.$$ The following result provides the conjugate 
formula for  the pointwise maximum of finitely many convex functions. It therefore  slightly generalizes (at least in the finite dimensional case) the results established for the case of two functions in \cite{FiS 00}  and alternatively proven in \cite{BoW 08}. The well-known subdifferential formula is also derived.

\begin{proposition}  For $f_1,\dots,f_m\in \Gamma_0(\bE)$ define $f:=\displaystyle\max_{i=1,\dots,m} f_i$. Then $f\in \Gamma_0(\bE)$ with  $\dom f=\bigcap_{i=1}^m \dom f_i$, 
\[
f^*(x)=\min_{v\in\Delta_m}\bigg(\sum_{i=1}^mv_if_i\bigg)^*(x)
\]
and 
\[
\p f(\bar x)= \co \set{\p f_i(\bar x)}{i\in I(\bar x)},
\]
where $I(\bar x)=\set{i}{f_i(\bar x)=f(\bar x)}$.
\end{proposition}
\begin{proof}
Define $F:\bE\to (\R^m)^\bullet$ by 
\[
F(x)= \begin{cases}
(f_1(x), \ldots, f_m(x))&\text{ if } x\in \bigcap_{i=1}^m\dom f_i,\\
+\infty_\bullet &\text{ otherwise},
\end{cases}
\] 
and $g\colon\R^m\to\R$ by   $g(v)=\max_{i=1,\ldots, m}v_i$. Then  $f=g\circ F$ (with the conventions made in Section \ref{sec:Prelim}) and we observe that  $F$ is $\R^m_{+}$-convex and $g$ is $\R^m_+$-increasing with $\dom g=\R^m$. Hence,  Corollary \ref{cor:Conj1}  is applicable with the qualification condition \eqref{eq:CQ1} trivially satisfied. Thus, for all $x\in\bE$,  we obtain
\[
\begin{aligned}
(g\circ F)^*(x)
&=\min_{v\in \R^m_+}g^*(v)+\ip{v}{F}^*(x)\\
&=\min_{v\in \R^m_+}\delta_{\Delta_m}(v)+\ip{v}{F}^*(x)\\
&=\min_{v\in\Delta_m}\bigg(\sum_{i=1}^mv_if_i\bigg)^*(x)
\end{aligned}
\]
where the second equality follows from \cite[Example~4.10]{Beck17}. Moreover 
\begin{eqnarray*} 
\p f(\bar x) & = & \bigcup_{v\in \p g(F(\bar x))} \p \ip{v}{F}(\bar x)\\
& = & \bigcup_{v\in \co \set{e_i}{i\in I(\bar x)}} \sum_{i=1}^m v_i\p f_i(\bar x)\\
& = & \co \set{\p f_i(\bar x)}{i\in I(\bar x)}.
\end{eqnarray*}
Here the second identity uses the known fact that $\p g(y)=\co \set{e_i}{y_i=g(y)}$ for all $y\in \R^m$.
\end{proof}

\subsection{Kiefer-Gaffke-Krafte inequality and the matrix-fractional function}\label{sec:Kiefer}
Recall that   $\bC^{n\times m}$ and $\bH^n$ are   the linear spaces of complex $n\times m$ and Hermitian $n\times n$ matrices, respectively.  Consider the  Euclidean space
\begin{equation}\label{eq:G}
\bG:=\bC^{n\times m}\times  \bH^n
\end{equation}
equipped with the inner product 
\begin{equation}\label{eq:ip for G}
\ip{\cdot}{\cdot}:\left((X,U),(Y,V)\right)\in \bG\times \bG\mapsto \Real\tr(Y^*X)+\Real\tr(VU).
\end{equation}
Define the mapping $\map{F}{\bG}{(\bH^n)^\bullet}$ by
\begin{equation}\label{eq:FLA}
F(X,V):=\left\{\begin{array}{rcl} X^*V^\dagger X  & \text{if} & \rge X\subset \rge V,\\
+\infty_\bullet, & \text{else}, \end{array}\right.
\end{equation}
where $X^*$ is the adjoint  of $X$ and $V^\dag$ is
the Moore-Penrose pseudoinverse of $V$, cf. \cite{HJ 13}.
Kiefer \cite{Kie59} has shown that $F$ is $\bH^n_+$-convex on the set
\[
\cD:=\set{(X,V)\in\bC^{n\times m}\times \bH^n_{++}}{\rge X\subset \rge V},
\]
and that   
\begin{eqnarray}
&F\left(\sum_{i=1}^k\lam_i(X_i,V_i)\right)=\sum_{i=1}^k\lam_iF(X_i,V_i)\quad \left(\lambda\in \Delta_k, (X_i,V_i)\in \cD\right)\label{eq:equality for KGK} \\
\Longleftrightarrow  &  X_i^*V_i^{-1}=X_i^*V_j^{-1}\quad (i,j\in\{1,\dots,k\}).
\end{eqnarray}
Gaffke and Krafte \cite{GaK} extend this result by 
showing that $F$ is convex on the set 
\[
\cl\cD=\set{(X,V)\in\bC^{n\times m}\times \bH^n_{+}}{\rge X\subset \rge V}
\]
again with respect to the cone $\bH^n_+$,
and \eqref{eq:equality for KGK} holds 
for a convex combination $\sum_{i=1}^k\lam_i(X_i,V_i)$ of points in $\cl\cD$
if and only if there is a matrix $C\in\bC^{n\times m}$ such that
$X_i=V_iC$ for $i=1,\dots,k$.

The $\bH^n_+$-convexity of the mapping $F$ from \eqref{eq:FLA} is  closely tied to the 
\emph{matrix-fractional function} \cite{BGH17, BGH 18, BuH15}
$\gamma:\widehat\bG\to\bR\cup\{+\infty\}$ given by 
\begin{equation}\label{eq:gamma}
\gamma (X,V):=\left\{\begin{array}{rcl } 
\half \tr\left(F(X,V)\right), &{\rm if} & (X,V)\in\widehat\cD,\\
                      +\infty ,&{\rm else,}                    
\end{array}\right.
\end{equation}
where $\widehat\bG:\bR^{n\times m}\times \bS^n$ and 
$\widehat\cD:=
\set{(X,V)\in \widehat\bG}{ V  \in\bS^n_+, \; \rge X\subset \rge V}$.
In \cite[Section 5.2]{BuH15}, 
it is shown that $\gamma$ is the
support function of the set 
\[
\cF:=\set{(Y,-\half YY^T)}{Y\in\bR^{n\times m}}
\]
and so $\gamma$ is a closed, proper,  convex (even sublinear) function 
with domain $\widehat\cD$. However, Theorem \ref{thm:K convexity} can
be used to establish a new and stronger result concerning the convexity of 
$\gamma$.

\begin{proposition}\label{prop:extended mff} Let $\bG$ be given by \eqref{eq:G},    $\map{F}{\bG}{(\bH^n)^\bullet}$ by \eqref{eq:FLA} and define $\map{\tilde\gamma}{\bG}{\bR\cup\{+\infty\}}$ by
\begin{equation}\label{eq:tgamma}
\tilde\gamma (X,V):=\left\{\begin{array}{rcl } 
\half \tr\left(F(X,V)\right), &{\rm if} & (X,V)\in\cl\cD,\\
                      +\infty ,&{\rm else.}                    
\end{array}\right.
\end{equation}
Then  $\tgam$ is a support function, in particular closed, proper, and convex. 
\end{proposition}
\begin{proof}
Let  $\bE_1:=\bG$ and $\bE_2:=\bH^n$  equipped with the inner products based on \eqref{eq:ip for G} and set $g:=\ip{\half I}{\cdot}\in \Gamma(\bE_2)$. Then $\tgam=g\circ F$.
Since $\half I\in \bH^n_{++}=-(\bH^n_{++})^\circ$, Theorem 
\ref{thm:K convexity} tells us that $\tilde\gamma$ is convex.
Obviously $\tgam$ is also proper and  positively homogeneous, i.e. $\tgam(\alpha(X,V))=\alpha\tgam(X,V)$ for any $\alpha\geq 0$.  We now show  that it is closed:
To this end, let $\{((X_k,V_k),\mu_k)\}\subset\epi{\tgam}$ and 
$((\tX,\tV),\tmu)\in \bG\times\R$ be such that
$((X_k,V_k),\mu_k)\rightarrow ((\tX,\tV),\tmu)$. We need to show that
$((\tX,\tV),\tmu)\in\epi{\tgam}$. Let $V_k$ have reduced singular-value
decomposition $V_k=U_kD_kU^*_k$ for each $k\in\bN$ so that
$U_k\in\bC^{n\times r_k}$ with $r_k:=\rank{V_k}\le n$,
$U^*_kU_k=I_{r_k}$ and $D_k=\diag(\sig_{k1},\dots,\sig_{kr_k})$ with $\sig_{kj}\in\R_{++}$
for $j=1,\dots,r_k$ and all $k\in\bN$. 
Set $\tilde r:=\rank \tV$. With no loss in generality, there is a 
$r\in\{\tilde r,\tilde r+1,\dots,n\}$ and $\tU\in\bC^{n\times r}$ such that
$r_k=r$ for all $k\in\bN$ with 
$\tU^*\tU=I_r$ and $U_k\rightarrow \tU$. Then
$D_k=U_k^*V_kU_k\rightarrow \tU^*\tV\tU=:\tD=\diag(\sig_1,\dots,\sig_r)$
for some $\sig_j\in\R_+,\ j=1,\dots,r$ satisfying $\sig_{kj}\rightarrow \sig_j,\ j=1,\dots,r$.
Since $\rge{X_k}\subset\rge{V_k}$,
there exists $C_k\in\bC^{r\times m}$ 
such that $X_k=U_kC_k$, for all $k\in \bN$, with
$C_k=U_k^*X_k\rightarrow \tU^*\tX=:\tC$.
Denote the columns of $C_k^*$ by $c_{kj},\ j=1,\dots,r,\ k\in\bN$,
and those of $\tC^*$ by $\tilde c_j$, $j=1,\dots,r$. Then
\[
0\le \sum_{j=1}^r\sig_{kj}^{-1}\sqnorm{c_{kj}}=\tr (C_k^*D_k^{-1}C_k)=2 \tgam(X_k,V_k)\le 2\mu_k.
\]
Consequently,
\[
0\le \limsup_{k\rightarrow\infty}\sum_{j=1}^r\sig_{kj}^{-1}\sqnorm{c_{kj}}\le 2\tmu.
\]
Hence, by passing to a subsequence if necessary, we can assume that
there exists $\eta_j\in\R_+$ such that $\sig_{kj}^{-1}\sqnorm{c_{kj}}\rightarrow \eta_j$
for $j=1,\dots,r$. Observe that if, for some $j_0\in\{1,\dots, r\}$, $\sig_{kj}\rightarrow 0$, then it must be the case that $c_{kj_0}\rightarrow 0$ since $\eta_{j_0}\in\R_+$. 
Therefore,
$\rge\tX\subset\rge \tV$, so that $(\tX,\tV)\in\dom \tgam$ with 
\[
2\tgam(\tX,\tV)=\sum_{j=1}^r\sig_j^\dagger\sqnorm{\tilde c_{j}}\le 
\sum_{j=1}^r\eta_j\le 2\tmu.
\]
That is, $\tgam$ is closed. Here $\sig_j^\dagger = \sig_j^{-1}$ if $\sig_j>0$ and $0$ otherwise. 

All in all, we have shown that $\tgam$ is closed, properm convex and positively homogeneous. The claim therefore follows  from H\"ormander's Theorem, see e.g. \cite[Corollary 13.2.1]{Roc 70}.
\end{proof}


\subsection{Variational Gram functions} \label{sec:VGF}
We equip   $\bS^n$  with the inner product $\ip{X}{Y}=\tr(XY)$.  Given a set $M\subset \bS^n_+$, the associated {\em variational Gram function (VGF)} \cite{BGH 18,JFX 17}  is given by 
\[
\Omega_{M}:\R^{n\times m} \to \rp,\quad \Omega_M(X)=\frac{1}{2}\sig_{M}(XX^T).
\]
Since $\sig_M=\sig_{\clco M}$ there is no loss in generality to assume that $M$ is closed and convex. For the remainder of this paragraph we let 
\begin{equation}\label{eq:F}
F:\R^{n\times m}\to \bS^n, \quad F(X)=\frac{1}{2}XX^T.
\end{equation}
Then $\Omega_M=\sig_M\circ F$ fits the composite scheme studied in Sections \ref{sec:KConv}-\ref{sec:Main}.  It is obvious that 
$\rge F=\bS_+^n$.    The following  lemma clarifies the $K$-convexity properties of $F$.
\begin{lemma}\label{lem:F} Let $F$ be given by \eqref{eq:F}. Then   $\bS^n_+$ is the smallest closed, convex cone  in $\bS^{n}$ with respect to which $F$ is convex. 

\end{lemma}
\begin{proof} Let $K$ be the smallest closed convex cone in $\bS^n$ such that $F$ is $K$-convex. On the one hand, it is easily verified that  $F$ is $\bS^n_+$-convex,  hence $K\subset\bS^n_+$. On the other hand, by \cite[Lemma 6.1]{Pen 99},
\[
(-K)^\circ= K_F:=\set{V\in \bS^n}{\ip{V}{F}\;\text{is convex}}.
\]
Now fixing $V\in \bS^n$ for all $X\in \R^{n\times m}$ the mapping $\nabla^2 \ip{V}{F}(X): \R^{n\times m}\times \R^{n\times m}\to \R$ is given by
\[
\nabla^2 \ip{V}{F}(X)[D,H]=\ip{V}{\frac{1}{2}(HD^T+DH^T)}.
\]
Clearly, this symmetric bilinear form is positive semidefinite  if and only if $V\in \bS_+^n$, which proves that 
$
K_F=\bS^n_+.
$
Finally, by taking the polar and using the fact that $(\bS^n_+)^\circ =-\bS^n_+$ we get 
\[
K = -(-K)^\circ = -K_F^\circ = -(\bS^n_+)^\circ =\bS^n_+.
\]
\end{proof}
 
 \noindent
 We now verify the remaining conditions necessary to apply our convex-composite framework from Section \ref{sec:Main} to $g:=\sigma_M$ and $F$ given by \eqref{eq:F}.

\begin{corollary}\label{cor:F} Let $M\subset \bS^n_+$ be nonempty,  closed and convex, and let  $F$ be given by \eqref{eq:F}. Then the following hold:
\begin{itemize}
\item[a)]  $-\hzn \sig_M\supset \bS^n_+$. 
\item[b)] $\sigma_M$ is $\bS^n_+$ increasing. 
\item[c)]  $M$ is bounded if and only if
\begin{equation}\label{eq:VGFCQ}
\rge F\cap \ri(\dom \sig_M)\neq \emptyset.
\end{equation}

\end{itemize}

\end{corollary}
\begin{proof}
a) We have $\overline{\cone}\;M\subset\bS^n_+$ as $M\subset\bS^n_+$, and hence, by \cite[Theorem 14.2]{Roc 70}
\[
-\hzn\sig_M = -(\overline{\cone}\;M)^\circ\supset-(\bS^n_+)^\circ =\bS^n_+.
\]
b) This follows from a) and  Corollary \ref{lem:HorizonIncrease}.
\smallskip

\noindent
c)  Since $M^\infty\subset \bS_+^n$, $M^\infty$ is pointed, hence by \cite[Exercise 6.22]{RoW 98},  we find that 
\[
\emptyset \neq \inter [(M^\infty)^\circ]=\set{W\in\bS^n}{\tr(WV)<0 \;(V\in M^\infty\setminus\{0\})}.
\]
Hence, by \cite[Corollary 14.2.1]{Roc 70}, we have 
\[
\ri(\dom \sig_M)=\ri[(M^\infty)^\circ]=\inter[(M^\infty)]^\circ
\]
 As $\rge F=\bS_+^n$, condition \eqref{eq:VGFCQ} is equivalent to  \[
\mathcal{F}:=\set{W\in\bS^n_+}{\tr(VW)<0 \; (V\in M^\infty\setminus\{0\})}\neq \emptyset.
\]
We claim that 
\[
\mathcal{F}:=\begin{cases}  \bS^n_+, & \text{if }  M \;\text{is bounded},\\
\emptyset,   & \text{else.} \end{cases}
\]
The first case is clear, since $M^\infty=\{0\}$ if (and only if) $M$ is bounded \cite[Theorem~3.5]{RoW 98}, in which case the condition restricting $\mathcal{F}$ is vacuous.

On the other hand, if $M$ is unbounded, then there exists $V\in M^\infty\setminus\{0\}\subset \bS^n_+\setminus \{0\}$. But then $\tr(VW)\geq 0$ for all $W\in \bS^n_+$, see e.g. \cite[Theorem~7.5.4]{HJ 13}, which proves the second case.

All in all, we have established the  desired equivalence.
\end{proof}

\noindent
We now combine our analysis in Section \ref{sec:KConv} and  \ref{sec:Main} with the matrix-fractional function  $\gamma$ from \eqref{eq:gamma}  to find a very  short proof of the conjugate function $\Omega_M^*$ in case $M$ is bounded (hence compact).   Here we note that  for $K:=\bS_+^n$ and 
\[
\Omega:=\set{(Y,W)\in \R^{n\times m}\times \bS^n}{\frac{1}{2}YY^T+W \preceq 0}
\]
we have  
\begin{equation}\label{eq:EpiOmega} 
\sig_{\Epi{K}{F}}(X,-V)= \sig_\Omega(X,V)\quad \left((X,V) \in \R^{n\times m}\times \bS^n\right).
\end{equation}

\noindent
Corollary \ref{cor:F} c) shows that our framework does not apply when $M$ is unbounded as the crucial condition \eqref{eq:CQ1} is violated then. 

The next result  covers  what was proven in \cite[Proposition 3.4]{JFX 17}  entirely  and one case of \cite[Proposition 5.10]{BGH 18}. 

\begin{proposition}\label{prop:VGFConj} Let $M\subset \bS^n_+$ be nonempty, convex and compact. Then  $\Omega^*_M$ is finite-valued and given by 
\begin{eqnarray*}
\Omega_{M}^*(X) & = &\frac{1}{2} \min_{V\in  M}\set{\tr(X^TV^\dagger X)}{\rge X\subset \rge V}.
\end{eqnarray*}
\end{proposition}
\begin{proof} $K:=\bS_+^n$. Recall that $(\clcone{M})^\circ=\hzn \sig_M$, see \cite[Theorem 14.2]{Roc 70}. Then  $F$ given by \eqref{eq:F} is $K$-convex  by    Lemma \ref{lem:F},  and $g$ is  $K$-increasing by Corollary \ref{cor:F} b).   By Corollary \ref{cor:F} c) we find that  condition \eqref{eq:CQ1simple} for $\sig_M\circ F$ and $K$ is satisfied  if (and only if) $M$ is bounded. Hence we compute that
\begin{eqnarray*}
\Omega_{M}^*(X) & = & \min_{V\in -K^\circ} \delta_M(V)+\ip{V}{F}^*(X)\\
& = &  \min_{V\in  M} \sig_{\Epi{K}{F}}(X,-V)\\
& = & \min_{V\in M} \sigma_{\Omega}(X,V)\\
& = & \min_{V\in M} \gamma(X,V)\\
& = & \min_{V\in M}\set{\frac{1}{2}\tr(X^TV^\dagger X)}{\rge X\subset \rge V}
\end{eqnarray*}
Here the first identity is due to Corollary \ref{cor:Conj1}, the second is  due to Lemma \ref{lem:Support}, the third uses \eqref{eq:EpiOmega}, the fourth follows from  \cite[Theorem 2]{BGH17} and the last identity is simply the definition of $\gamma$ in \eqref{eq:gamma}.

As $M$ is compact this proves also the finite-valuedness and thus concludes the proof.
\end{proof}

\noindent
A formula for the subdifferential of the VGF $\Omega_{M}$ is easily established as well.

\begin{corollary}\label{cor:VGFSub}  Let $M\subset \bS^n_+$ be nonempty, convex and compact. Then 
\[
\p \Omega_M(X)=\set{VX}{V\in \argmax_{M}\ip{XX^T}{\cdot}}.
\]
\end{corollary}
\begin{proof} First observe  that $F'(X)^*V=VX$ for all $X\in \R^{n\times m}$ and that $\p \sigma_M(U)=\argmax_{M}\ip{U}{\cdot}$ for all $U\in \bS^n$. Therefore,  Corollary \ref{cor:Conj1} (which is applicable for the same reasons as in Proposition \ref{prop:VGFConj}) and Remark \ref{rem:DiffScalar} 
\end{proof}

\subsection{Spectral function of symmetric matrices}\label{sec:Lewis}

Consider the  function
\begin{equation}\label{eq:Lambda}
\lambda\colon\bS^n\to\R^n,\; \lambda(X):=(\lambda_1(X), \ldots, \lambda_n(X)),\quad (X\in\bS^n)
\end{equation}
where $\lambda_1(X)\geq\ldots\geq\lambda_n(X)$ are the  eigenvalues of $X$.  In order to apply our framework from above, we equip  $\bS^n$ with the inner product 
\[
\ip{\cdot}{\cdot}:\bS^n\times \bS^n\to \R, \; \ip{X}{Y}=\tr(XY)
\]
and $\R^n$ with the  standard inner product.
The following results  are, even without any convexity assumptions,  originally due to Lewis \cite{Lew 95, Lew 96} and were recently confirmed by  a simplified proof due to Drusvyatskiy and Paquette   \cite{DrP 18}, however under an additional lower semicontinuity assumption. Our technique of proof, based on the convex-composite framework with $g\in \Gamma(\R^n)$ and $F:=\lambda$, differs substantially from the latter references, but shares some similarity with \cite[Section 9]{Pen 99} and we make use of some of the auxiliary results established there. Note that the latter reference does not establish the conjugate formula below.

\begin{proposition}\label{prop:Spec}
\label{p:spectral} Let  $g\in\Gamma(\R^n)$ be  permutation invariant (i.e., its value does not change by reordering of the  argument vector) and let $\lambda:\bS^n\to\R^n$ be given by \eqref{eq:Lambda}. Then  $g\circ \lambda$ is convex  and  
\[
(g\circ\lambda)^*=g^*\circ\lambda.
\]
%
\end{proposition} 
\begin{proof}

Consider the closed, convex cone
\begin{equation}\label{eq:KSpec}
K=\set{v\in\R^n}{\sum_{i=1}^kv_i\geq 0, k= 1, \ldots, n-1, \sum_{i=1}^nv_i=0},
\end{equation}
whose polar cone is 
\begin{equation}\label{eq:PolarKSpec}
K^\circ = \set{v\in\R^n}{v_1\leq \ldots\leq v_n}. 
\end{equation}
By \cite[Theorem~6.5]{Roc 70} we have 
\begin{equation}\label{eq:RintKSpec}
\ri K=\set{v\in\R^n}{\sum_{i=1}^kv_i> 0, k= 1, \ldots, n-1, \sum_{i=1}^nv_i=0}.
\end{equation}
We first note that $\dom\lambda =\bS^n$ and hence, $\lambda(\ri(\dom\lambda))=\rge\lambda=-K^\circ$. 
Next, by \cite[Corollary~9.3]{Pen 99}, $\lambda$ is $K$-convex and for any $v\in -K^\circ$, $\ip{v}{\lambda}=\sig_{\Omega_v}$ where $\Omega_v=\set{X\in\bS^n}{v-\lambda(X)\in K}$ is a closed, convex set in $\bS^n$.  In particular, $\lambda$ is $K$-convex,  and  $\ip{v}{\lambda}^*=\delta_{\Omega_v}$.  Moreover, observe that since $g$ is permutation invariant,  so is $g^*$, and hence, by \cite[Lemma~9.5]{Pen 99}, both $g$ and $g^*$ satisfy \eqref{eq:FKMonotone} with  $F=\lambda$ and $K$ from \eqref{eq:KSpec},  which shows in particular that $g\circ \lambda$ is convex, see  Proposition \ref{prop:composite1}.
Now set $G:=\set{v\in\R^n}{v_1=\ldots = v_n}$ and  consider two cases:
\smallskip

\noindent
\underline{$\dom g=G$:} For $X\in\bS^n$ observe that 
\[
\begin{aligned}
(g\circ\lambda)^*(X)
&=\sup_{\lambda(Y)\in\dom g}\left\{\ip{X}{Y}-g(\lambda(Y))\right\}\\
&=\sup_{\alpha\in\R}\left\{\ip{X}{\alpha I_n}-g(\alpha, \ldots,\alpha)\right\}\\
&=\sup_{\alpha\in\R}\left\{\ip{(\alpha, \ldots,\alpha)}{(\lambda_1(X), \ldots, \lambda_n(X))}-g(\alpha, \ldots,\alpha)\right\}\\
&=\sup_{y\in \dom g}\left\{\ip{y}{\lambda(X)}-g(y)\right\}\\
&=g^*(\lambda(X)).
\end{aligned}
\]
Here the second identity uses that $\dom g=G$ and that  $\lambda(Y)\in G$ if and only if $Y$ is a multiple of $I_n$. The third is due to the fact that $\ip{X}{\alpha I_n}=\alpha \tr(X)=\alpha\sum_{i=1}^m\lambda_i(X)$. The fourth one uses again that $\dom g=G$. This proves the conjugate formula in this case.
\smallskip

\noindent
\underline{$\dom g\neq G$:}  Here we want to apply Theorem \ref{th:ConjMain}  to $g$,  $F=\lambda$ and $K$ given in \eqref{eq:KSpec}. We already established above that  \eqref{eq:FKMonotone} holds.  We now verify the qualification condition~\eqref{eq:CQ1} which in the current setting  reads
\begin{equation}
\label{eq:cq3e}
\emptyset\neq \ri(\dom g -K)\cap(-K^\circ)=(\ri(\dom g)-\ri K)\cap(-K^\circ).
\end{equation}
Fix $v\in\ri(\dom g)$. Denote by $S_n$ the set of all permutations $\sigma\colon\{1, \ldots, n\}\to\{1, \ldots, n\}$ and the action of such a permutation on $v$ by $\sigma(v)=(v_{\sigma(1)},\ldots, v_{\sigma(n)})$. Since $g$ is unchanged under the reordering of its argument vector, $\sigma(v)\in\dom g$ for all $\sigma\in S_n$ and hence by \cite[Theorem~6.1]{Roc 70}, $\bar v=\frac{1}{n!}\sum_{\sigma\in S_n}\sigma(v)\in\ri(\dom g)$. By assumption, $\dom g\neq G$, hence there exists  $\hat v\in\dom g$ be such that at least two its components are distinct. Let $\sigma_*\in S_n$ be such that $\sigma_*(\hat v)$ is a nondecreasing reordering of $\hat v$. By replacing $v$ by $\frac{1}{2}\bar v+\frac{1}{2}\sigma_*(\hat v)\in\ri(\dom g)$ (by \cite[Theorem~6.1]{Roc 70}), we can assume that $v_1\geq\ldots\geq v_k>v_{k+1}\geq \ldots\geq v_n$ for some $1\leq k<n$.  
Now,  pick $0<\alpha <\big(1- \frac{k}{n}\big)(v_k-v_{k+1})$ and define $b\in \R^n$ component-wise by
\[
b_i=
\begin{cases}
\alpha,&\text{ if } 1\leq i\leq k,\\
-\frac{k\alpha}{n-k},&\text{ if } k<i\leq n
\end{cases}\quad \quad (i=1,\dots,n).
\]
Then we have  
\[
\sum_{i=1}^\ell b_i=k\cdot\alpha-\frac{\ell-k}{n-k}k\cdot\alpha\quad(\ell =1,\dots,n),
\]
hence $b\in \ri K$, see \eqref{eq:RintKSpec}. Moreover, we have 
\[
(v_i-b_i)-(v_{i+1}-b_{i+1})=\left\{\begin{array}{rcl}
  v_k-v_{k+1}-\frac{\alpha\cdot n}{n-k}, & \text{if} & i=k, \\
 v_i-v_{i+1}, &  \text{else}
\end{array}\right\}\geq 0\quad (i=1,\dots,n),
\]
and therefore $v-b\in -K^\circ$, cf. \eqref{eq:PolarKSpec}.
 All in all,  Theorem \ref{th:ConjMain} is applicable and yields
\[
\begin{aligned}
(g\circ \lambda)^*(X)
&=\min_{v\in -K^\circ}\; g^*(v)+\ip{v}{\lambda}^*(X)\\
&=\min_{v\in -K^\circ}\; g^*(v)+\delta_{\Omega_v}(X)\\
&=\min_{v\in -K^\circ:\; X\in\Omega_v}\; g^*(v)\\
&=g^*(\lambda(X)),
\end{aligned}
\]
where the last identity follows from the fact that $-K^\circ\ni\lambda(X)\leq_Kv$ and $g^*$ satisfies \eqref{eq:FKMonotone} as seen above.
%
\end{proof}

\begin{remark} Using  the same notation as in the proof of the Proposition~\ref{p:spectral}, we have the following comments:
\begin{itemize}
\item[a)]  If $b=(b_1, \ldots, b_n)\in\ri K$ then $b_n<0< b_1$, and hence, we can deduce that the qualification condition~\eqref{eq:CQ1} is equivalent to the fact that $\dom g\neq G$. 
\item[b)] Analogous   \cite{Lew 95, Lew 96},  the formula for the subdifferential of $g\circ\lambda$ can be obtained using the conjugacy result in Proposition~\ref{p:spectral} and the von Neumann's inequality \cite{neu 35} 
\[
\ip{X}{Y}\leq \ip{\lambda(X)}{\lambda(Y)}\quad (X,Y\in \bS^n).
\]
To this end, observe that
\[
\begin{aligned}
\partial (g\circ\lambda)(X)
&=\set{Y\in\bS^n}{(g\circ\lambda)(X)+(g\circ\lambda)^*(Y)=\ip{X}{Y}}\\
&=\set{Y\in\bS^n}{g(\lambda(X))+g^*(\lambda(Y))=\ip{\lambda(X)}{\lambda(Y)}}\\
&=\set{Y\in\bS^n}{\lambda(Y)\in\partial g(\lambda(X))}\\
&=\set{U^*\mathrm{diag}(v)U}{v\in\partial g(\lambda(X)),\; U^*U=I_n},
\end{aligned}
\]
where the second equality comes from the Fenchel-Young inequality \eqref{eq:FYI} applied to $g$.
\end{itemize}
\end{remark}

\noindent
We close this section with an example that shows that Proposition \ref{prop:Spec} cannot be derived from the results by Combari et al. \cite{CLT 94} as their assumptions are not met.

\begin{example}\label{ex:Spec}
In Proposition~\ref{p:spectral}, set $g=\delta_{\{0\}}\in\Gamma(\R^n)$ which is permutation invariant and satisfies \eqref{eq:FKMonotone} but  is not increasing w.r.t. $K$ given in \eqref{eq:KSpec}. Moreover,  
$
\R_+(\dom g-\rge\lambda)=K^\circ
$
is not a  subspace.

\hfill $\diamond$
\end{example}

\subsection{A Farkas-type result}
\label{sect:Fakas} 
In this section, we extend the  Farkas-type result shown in \cite[Theorem~4.1]{Bot et al. 06}. The latter reference uses a Lagrangian duality framework  where our proof  is based on  Corollary~\ref{cor:AddConj} and Theorem \ref{th:ConjMain}. 
For this purpose, let $X\subset\bE_1$ be a nonempty, convex set, let $K, L\subset\bE_2$  be nonempty, closed, convex cones,  let $F\colon\bE_1\to\bE_2^\bullet$ be $K$-convex, let  $G\colon\bE_1\to\bE_2^\bullet$ be $L$-convex, $f\in\Gamma(\bE_1)$,  and let $g\in\Gamma(\bE_2)$ satisfy \eqref{eq:FKMonotone}. Consider the following qualification condition:
\begin{equation}
\label{eq:CQfarkas}
\exists\; \bar x\in\ri X\cap\ri(\dom f)\cap\ri(\dom F)\cap\ri(\dom G):\quad
\begin{cases}
F(\bar x)\in\ri(\dom g-K),\\
G(\bar x)\in \ri(-L).
\end{cases}
\end{equation}
We point out that a particular version of the following result with $f=0$, $F$ and $G$ are finite-valued, $K=L=\R^n_+$, and $g$ is assumed to be $K$-increasing instead of satisfying \eqref{eq:FKMonotone}, is proved in \cite{Bot et al. 06}.

\begin{theorem}
\label{th:farkas}
Under \eqref{eq:CQfarkas}, the following are equivalent:
\begin{itemize}
\item[a)] $x\in X$, $G(x)\in -L\;\Longrightarrow\; f(x)+g(F(x))\geq 0$.
\item[b)] There exist $\bar y, \bar z, \bar s\in\bE_1$, $\bar u\in\ -K^\circ$, $\bar v\in -L^\circ$ such that
\[
g^*(\bar u)+\ip{\bar u}{F}^*(\bar y)+\ip{\bar v}{G}^*(\bar z)+\sigma_X(\bar s) +f^*(-\bar y-\bar z-\bar s)\leq 0.
\]
\end{itemize}
\end{theorem}
\begin{proof} Consider $\bE=\bE_2\times\bE_2$ equipped with the inner product $\ip{(x_1,y_1)}{(x_2,y_2)}:=\ip{x_1}{x_2}+\ip{y_1}{y_2}$ and define $\tilde K:=K\times L$ and $\tilde f:=f+\delta_X$. Let  $\tilde F\colon x\in \bE_1\mapsto (F(x), G(x))\in \bE$ with $\dom\tilde F=\dom F\cap\dom G$, and $\tilde g\colon(u,v) \in \bE\mapsto g(u)+\delta_{-L}(v)$ with $\dom\tilde g=\dom g\times (-L)$ and $\tilde{g}^*\colon (u,v)\mapsto g^*(u)+\delta_{-L^\circ}(v)$. We then see that $\tilde F$ is $\tilde K$-convex and since $g\in \Gamma(\bE_2)$ satisfies \eqref{eq:FKMonotone} and $\delta_{-L}$ is $L$-increasing, $\tilde g\in \Gamma(\bE)$ satisfies \eqref{eq:FKMonotone}. In addition, we observe that  \eqref{eq:CQfarkas} is equivalent to \eqref{eq:CQ2}.

\noindent 
a) $\Rightarrow$ b): We note that
\begin{equation}
\label{eq:farkas3b}
0\leq \inf\limits_{\overset{x\in X}{G(x)\in -L}}\;f(x)+g(F(x))=\inf_{\bE_1}\;f+\delta_X+g\circ F+\delta_{-L}\circ G=\inf_{\bE_1}\;\tilde f+\tilde g\circ\tilde F=-(\tilde f+\tilde g\circ\tilde F)^*(0),
\end{equation}
and we conclude that $(\tilde f+\tilde g\circ\tilde F)^*(0)\leq 0$. We now apply Corollary~\ref{cor:AddConj} a) with $\tilde K$, $\tilde f$, $\tilde g$ and $\tilde F$ to get
\begin{equation}
\label{eq:farkas3}
\begin{aligned}
(\tilde f+\tilde g\circ\tilde F)^*(0)
&=\min\limits_{\overset{(u,v)\in -K^\circ\times -L^\circ}{y\in\bE_1}}\; (f+\delta_X)^*(y)+ g^*(u)+\delta_{-L^\circ}(v)+\big(\ip{u}{F}+\ip{v}{G}\big)^*(-y)\\
&=\min_{(u,v)\in -K^\circ\times -L^\circ}\; g^*(u)+\min_{y\in\bE_1}\big\{(f+\delta_X)^*(y)+\big(\ip{u}{F}+\ip{v}{G}\big)^*(-y)\big\}\\
&=\min\limits_{\overset{(u,v)\in -K^\circ\times -L^\circ}{y, z, s\in\bE_1}}\; g^*(u)+\ip{u}{F}^*(y)+\ip{v}{G}^*(z)+\sigma_X(s)+f^*(-y-z-s),
\end{aligned}
\end{equation}
where the   last identity  follows from  successively applying  Theorem~\ref{th:AttBre} which is possible due to  \eqref{eq:CQfarkas}. The desired implication then follows from \eqref{eq:farkas3b} and \eqref{eq:farkas3} (and the fact that the minimum is taken).

\noindent
b) $\Rightarrow$ a): We observe that
\begin{equation}
\label{eq:farkas3a}
\begin{aligned}
0&\geq\inf\limits_{\overset{(u,v)\in -K^\circ\times -L^\circ}{y, z, s\in\bE_1}}\; g^*(u)+\ip{u}{F}^*(y)+\ip{v}{G}^*(z)+\sigma_X(s)+f^*(-y-z-s)\\
&=\inf\limits_{(u,v)\in -K^\circ\times -L^\circ}\;g^*(u)+(f^*\Box\sigma_X\Box\ip{u}{F}^*\Box\ip{v}{G}^*)(0)\\
&=\inf\limits_{(u,v)\in -K^\circ\times -L^\circ}\;g^*(u)+((f+\delta_X)^*\Box(\ip{u}{F}+\ip{v}{G})^*)(0)\\
&=\inf\limits_{(u,v)\in -K^\circ\times -L^\circ}\;g^*(u)+\inf_{y\in\bE_1}\big\{(f+\delta_X)^*(y)+\big(\ip{u}{F}+\ip{v}{G}\big)^*(-y)\big\}\\
&=\inf\limits_{y\in\bE_1}\;(\tilde f)^*(y)+\inf_{(u,v)\in -K^\circ\times -L^\circ}\big\{g^*(u)+\delta_{-L^\circ}(v)+\big(\ip{u}{F}+\ip{v}{G}\big)^*(-y)\big\}\\
&\geq\inf\limits_{y\in\bE_1}\;(\tilde f)^*(y)+(\tilde g\circ\tilde F)^*(-y)\\
&\geq -\inf_{x\in\bE_1}\tilde f(x)+\tilde g(\tilde F(x))\\
&=-\inf_{x\in X, G(x)\in -L}\;f(x)+g(F(x)),
\end{aligned}
\end{equation}
where the  third line  repeatedly uses Theorem \ref{th:AttBre} and \eqref{eq:CQfarkas},  
the sixth line follows from Theorem~\ref{th:ConjMain}~a), and the seventh is due to \eqref{eq:FYI}. The desired implication now follows from \eqref{eq:farkas3a}.  
\end{proof}

\begin{remark} We have the following observations:
\begin{itemize}
\item[a)] 
The result and the proof of Theorem~\ref{th:farkas} can be extended to the case where $f$ and the convex  convex-composite function are finite sums of functions of the same type, respectively. 
\item[b)]  One can also state Theorem~\ref{th:farkas} as a theorem of the alternative:
 Either
\[
\begin{cases}
x\in X, G(x)\in -L,\\
 f(x)+g(F(x))<0
 \end{cases}
\quad\text{ or }\quad 
\begin{cases}
u\in -K^\circ, v\in -L^\circ, y, z, s\in\bE_1,\\
g^*(u)+\ip{u}{F}^*(y)+\ip{v}{G}^*(z)+\sigma_X(s)+f^*(-y-z-s)\leq 0
 \end{cases}
 \]
 has a solution (but not both). 
 
 \end{itemize}
\end{remark}

\section{Final remarks} 

\noindent
In this note we developed a full conjugacy and subdifferential calculus for the important subclass of  convex-composite functions that are fully convex. Working  in the finite-dimensional setting, we took advantage of the  relative topology of convex sets, which promoted   infimal convolution  as our main  tool. Combined with the concept of $K$-convexity  this enabled us to prove the central convex-analytic results for convex convex-composite functions under weak assumptions (no lower semicontinuity and  weak monotonicity properties) on the functions in play, using a  verifiable, point-based Slater-type qualification condition.     A myriad of applications underlined  the versatility of our findings and the unifying strength of the convex-composite setting even in the  fully convex case. In particular, we were able to prove a new extension for the matrix-fractional function to the complex domain, and obtain new and simple  proofs for the conjugate and subdifferential of variational Gram functions, an alternative proof of Lewis'  result on  spectral maps in the convex setting, as well as a refined version of a  Farkas-type result due to Bot et al. 

For future research we plan on bringing applying  the convex-analytic results established here to  numerical methods for solving convex convex-composite optimization problems.

\section*{Acknowledgments.}
The first author was partially supported by  NSF grant DMS­1514559. The second author was partially supported by NSERC discovery grant RGPIN-2017-04035. The third author was partially supported by a grant from  {\em Centre de recherches de math\'ematiques (CRM)}, Montr\'eal.



\begin{thebibliography}{99}



\bibitem{Beck17}
{\sc A.~Beck:} {\em First-Order Methods in Optimization.} MOS-SIAM Series on Optimization, 25. 
  Society for Industrial and Applied Mathematics (SIAM), Philadelphia, PA; Mathematical Optimization Society, Philadelphia, PA, 2017.
  
  \bibitem{Dim 71}
{\sc D.~Bertsekas:} {\em Control of Uncertain Systems with a Set-Membership Description of Uncertainty.} MOS-PhD Thesis, Massachusetts Institute of Technology, 1971.
  
\bibitem{Bor 74}{\sc J. M. Borwein:}
{\em Optimization with respect to Partial Orderings.}
Ph.D. Thesis, University of Oxford, 1974.

\bibitem{Bot et al. 06} {\sc R.I.~Bo\c{t}, I.B.~Hodrea, and G.~Wanka:}
{\em Farkas-type results for inequality systems with composed convex functions via conjugate duality.}
Journal of Mathematical Analysis and Applications 322(1), 2006, pp.~316--328.

\bibitem{BGW 07}{\sc R.I.~Bo\c{t}, S.-M.~Grad, and G.~Wanka:} {\em New constraint qualification and conjugate duality for composed convex optimization problems.} Journal of Optimization Theory and Applications 135(2), 2007, pp. 241--255.

\bibitem{BoW 08}{\sc R.I.~Bo\c{t}  and G.~Wanka:}
{\em The conjugate of the pointwise maximum of two convex functions revisited.}
Journal of Global Optimization 41(3),  2008,  pp.~625--632.


\bibitem{BGW 08}{\sc R.I.~Bo\c{t}, S.-M.~Grad, and G.~Wanka:} {\em A new constraint qualification for the formula of the subdifferential of composed convex functions in infinite dimensional spaces.} Mathematische  Nachrichten 281(8), 2008, pp. 1088--1107.


\bibitem{BGW 09}{\sc R.I.~Bo\c{t}, S.-M.~Grad, and G.~Wanka:}
{\em Generalized Moreau-Rockafellar results for composed convex functions.}
Optimization  58(7),  2009, pp.~917--933.

\bibitem{BJW 06}{\sc R.S.~Burachik, V.~Jeyakumar, and Z.-Y.~Wu:}
{\em Necessary and sufficient conditions for stable conjugate duality.}
Nonlinear Analysis  64(9),  2006, pp.~1998--2006.

\bibitem{burke1985descent}{\sc J.V.~Burke:} 
{\em Descent methods for composite nondifferentiable optimization
  problems.}
Mathematical Programming 33(3), 1985, pp.~260--279.

\bibitem{burke1987second}{\sc J.V.~Burke:}
{\em Second order necessary and sufficient conditions for convex
  composite {NDO}.}
  Mathematical Programming 38(3), 1987, pp.~287--302.
  
\bibitem{Bur91}{\sc J.V.~Burke:} 
{\em An exact penalization viewpoint of constrained optimization.}
SIAM Journal on Control and Optimization 29(4), 1991, pp.~968--998.

\bibitem{BuE19}{\sc J.V.~Burke and A.~Engle:}
{\em Strong metric (sub)regularity of KKT mappings for piecewise linear-quadratic convex-composite optimization and the quadratic convergence of Newton's method.}  Mathematics of Operations Research, to appear.

\bibitem{BF 95}{\sc J. V. Burke and M. C. Ferris:}
{\em A Gauss-Newton method for convex composite optimization.}
Mathematical Programming 71(2), 1995, pp.~179--194.

\bibitem{BGH17} {\sc J.~V.~Burke, Y.~Gao and T.~Hoheisel:}
{\em Convex geometry of the generalized matrix-fractional function.}
SIAM Journal on Optimization 28(3), 2018, pp.~2189--2200.


\bibitem{BGH 18} {\sc J.~V.~Burke, Y.~Gao, and T.~Hoheisel:} {\em  Variational properties of matrix functions via the generalized matrix-fractional function.}  SIAM Journal on Optimization  29(3), 2019, pp.~1958--1987.

\bibitem{BuH 13}{\sc J.~V.~Burke and T.~Hoheisel:}
{\em  Epi-convergent smoothing with applications to convex composite functions.} 
 SIAM Journal on Optimization 23(3), 2013, pp. 1457--1479.

\bibitem{BuH15}
{\sc J.~V.~Burke and T.~Hoheisel:} {\em Matrix support functionals for inverse
  problems, regularization, and learning.} SIAM Journal on Optimization  25(2),
  2015, pp.~1135--1159.


\bibitem{burke1992optimality}{\sc J.~V.~Burke and R.~Poliquin:}
{\em Optimality conditions for non-finite valued convex
  composite functions.} 
  Mathematical Programming 57(1), Ser. B, 1992, pp.~103--120.


\bibitem{cibulka2016strong}{\sc R.~Cibulka, A.~Dontchev, and A.~Kruger:}
{\em Strong metric subregularity of mappings in variational analysis and optimization.}
Journal of Mathematical Analysis and Applications 457(2), 2018, pp.~1247--1282.
  
  
\bibitem{Cla 83} {\sc F.H.~Clarke:}
\textit{Optimization and Nonsmooth Analysis.}
John Wiley \& Sons, New York 1983.
  

\bibitem{CLT 94}{\sc C.~Combari, M.~Laghdir, and L.~Thibault:}
{\em Sous-diff\'erentiels de fonctions convexes compos\'ees.}
  Annales des Sciences Math\'ematiques du Qu\'ebec 18(2),  1994, pp.~119--148.

\bibitem{cui2018composite}{\sc Y.~Cui, J.S.~Pang, and B.~Sen:}
{\em Composite difference-max programs for modern
  statistical estimation problems.} 
 SIAM Journal on Optimization 28(4), 2018, pp.~3344--3374.

\bibitem{davis2018stochastic}{\sc D.~Davis D and D.~Drusvyatskiy:}
{\em Stochastic model-based minimization of weakly
  convex functions.} 
 SIAM Journal on Optimization 29(1), 2019, pp.~207--239.
%

\bibitem{deng1996uniqueness}{\sc S.~Deng:}
{\em On uniqueness of Lagrange multipliers in composite optimization.} 
 Journal of Mathematical Analysis and Applications 201(3), 1996, pp.~689--696.


\bibitem{drusvyatskiy2018error}{\sc D.~Drusvyatskiy D and A.S.~Lewis:}
{\em Error bounds, quadratic growth, and linear convergence of proximal methods.}
Mathematics of Operations Research 43(3), 2018, pp.~919--948.


\bibitem{drusvyatskiyefficiency}{\sc D.~Drusvyatskiy and C.~Paquette:}
{\em Efficiency of minimizing compositions of convex functions and smooth maps.}
Mathematical Programming, 2018, \url{https://doi.org/10.1007/s10107-018-1311-3}.


\bibitem{DrP 18}{\sc D.~Drusvyatskiy and C.~Paquette:}
{\em Variational analysis of spectral functions simplified.}
Journal of Convex Analysis 25(1), 2018, , pp.~119--134.




\bibitem{DuR2017arXiv170502356D}{\sc J.C.~Duchi and F.~Ruan:}
{\em Solving (most) of a set of quadratic equalities: Composite optimization for robust phase retrieval}.
 Information and Inference: A Journal of the IMA 8(3), 2019, pp.~471--529.



\bibitem{duchi2017stochastic}{\sc J.C.~Duchi and F.~Ruan:}
{\em Stochastic methods for composite and weakly convex optimization problems.}
SIAM Journal on Optimization 28(4), 2018, pp.~3229--3259.


\bibitem{FiS 00}{\sc S.P.~Fitzpatrick and S.~Simons:}
{\em On the pointwise maximum of convex functions.}
Proceedings of the American Mathematical Society 128(12), 2000, pp.~3553--3561.



\bibitem{HiH 06} {\sc J.-B.~Hiriart-Urruty:}
{\em A note on the Legendre-Fenchel transform of convex composite functions.}
In {\em Nonsmooth Mechanics and Analysis.} Eds.  P. Alart, O. Maisonneuve, and R. T. Rockafellar,
Springer, 2006, pp.~35--46.


\bibitem{HUL 01}{\sc J.-B.~Hiriart-Urruty and C.~Lemar\'echal:}
{\em Fundamentals of Convex Analysis.}
Springer-Verlag, New York, N.Y., 2001.

\bibitem{GaK}{\sc N.~Gaffke and O.~Kraffte:}
{\em Matrix inequalities in the L\"{o}wner ordering.} 
In: B. Korte (Ed.), Modern Applied Mathematics, North-Holland, Amsterdam-New York, 1982.

\bibitem{HJ 13}
{\sc R. Horn and C. R. Johnson:} {\em Matrix analysis (2ed).} Cambridge University Press, Cambridge, 2013.


\bibitem{JFX 17} {\sc A.~Jalali, M.~Fazel, and L.~Xiao:}
{\em Variational Gram functions: convex analysis and optimization.}
SIAM Journal on Optimization 27(4), 2017, pp.~2634--2661. 

\bibitem{kawasaki1988second}{\sc H.~Kawasaki:}
{\em Second-order necessary conditions of the Kuhn-Tucker type
  under new constraint qualifications.}
 Journal of Optimization Theory and Applications 57(2), 1988, pp.~253--264.
 
 \bibitem{Kie59}{\sc J.~Kiefer:}
{\em Optimum experimental designs.}
Journal of the Royal Statistical Society, Series B, Methodological 21, 1959, pp.~272--319.
 
 \bibitem{KK95}{\sc A.G.~Kusraev and S.S.~Kutateladze:}
 {\em Subdifferentials: Theory and Applications.} 
 Mathematics and its Applications, 323. Kluwer Academic Publishers Group, Dordrecht, 1995.

\bibitem{Lew 95} {\sc A.S.~Lewis:}
{\em The convex analysis of unitarily invariant matrix functions.}
Journal of Convex Analysis 2(1--2), 1995, pp.~173--183.

\bibitem{Lew 96} {\sc A.S.~Lewis:}{\em Convex analysis on the hermitian Matrices}
SIAM Journal on  Optimimization  6(1), 1996, pp.~164--177. 



\bibitem{lewis2016proximal}{\sc A.S.~Lewis and S. J.~Wright:} 
{\em A proximal method for composite minimization.}
Mathematical Programming 158(1-2), Series B, 2016, pp.~501--546.


\bibitem{neu 35}{\sc J. von Neumann:} 
{\em Some matrix-inequalities and metrization of matrix-space.} 
Tomsk. Univ. Rev., Vol 1, pp.~286--300, 1937. Reprinted in Collected Works (Pergamon Press, 1962), iv, 205--219.
	
\bibitem{Pen 99}{\sc T.~Pennanen:}
{\em Graph-convex mappings and K-convex functions.}
Journal of Convex Analysis 6(2), 1999, pp.~235--266.

\bibitem{powell1978algorithms}{\sc M.J.D.~Powell:}
{\em Algorithms for nonlinear constraints that use Lagrangian functions.}
Mathematical programming 14(2), 1978, pp.~224--248.

\bibitem{powell1978fast}{\sc M.J.D.~Powell:}
{\em A fast algorithm for nonlinearly constrained optimization calculations.}
Lecture Notes in Mathematics, Vol.~630, Springer, Berlin, 1978.


%

\bibitem{Roc 70}
{\sc R.T.~Rockafellar:} {\em Convex Analysis.} Princeton Mathematical Series, No. 28. Princeton University Press, Princeton, N.J. 1970.

\bibitem{rockafellar1985extensions}{\sc R.T.~Rockafellar:} 
{\em Extensions of subgradient calculus with applications to optimization.} 
Nonlinear Analysis: Theory, Methods \& Applications 9(7), 1985, pp.~665--698.

\bibitem{rockafellar1989second}{\sc R.T.~Rockafellar:} 
{\em Second-order optimality conditions in nonlinear
  programming obtained by way of epi-derivatives.} 
  Mathematics of Operations Research 14(3), 1989, pp.~462--484.
  




\bibitem{RoW 98}
{\sc R.T.~Rockafellar and R.J.-B.~Wets:} {\em Variational Analysis.}
  Grundlehren der Mathematischen Wissenschaften, Vol.~317, Springer-Verlag, Berlin, 1998.

\bibitem{ShS 03}{\sc  A.~Shapiro and K.~Scheinberg:}{\em Duality and optimality conditions.}  In: Handbook of Semidefinite Programming: Theory, Algorithms, and Applications. Eds.: H.~Wolkowicz, R.~Saigal, L.~Vandenberghe. Springer Science+Business Media, Springer, New York, 2003.

\bibitem{wright1987local}{\sc S.J.~Wright:}
{\em Local properties of inexact methods for minimizing nonsmooth
  composite functions.} 
  Mathematical programming 37(2), 1987, pp.~232--252.

\bibitem{yuan1985superlinear}{\sc Y.~Yuan:}
{\em On the superlinear convergence of a trust region algorithm for
  nonsmooth optimization.} 
  Mathematical Programming 31(3), 1986, pp.~269--285.
  
  \bibitem{Zal 83}{\sc C.~Z\u{a}linescu:}
{\em Duality for vectorial nonconvex optimization by convexification and applications.} 
 Analele \c{S}tiin\c{t}ifice ale Universit\u{a}\c{t}ii  "Al. I. Cuza'' din Ia\c{s}i. Sec\c{t}iunea I a Matematic\u{a} 29(3), 1983, pp.~15--34.


\end{thebibliography}


\end{document}